\documentclass[letter, 11pt]{article}
\usepackage[english]{babel}
\usepackage[utf8]{inputenc}

\usepackage[margin = 1.2 in, top = 1 in, bottom = 1.2 in]{geometry}
\makeatletter
\g@addto@macro\normalsize{%
  \setlength\abovedisplayskip{7pt}
  \setlength\belowdisplayskip{7pt}
  \setlength\abovedisplayshortskip{7pt}
  \setlength\belowdisplayshortskip{7pt}
}
\makeatother

\usepackage{graphicx}
\usepackage{cancel}

\usepackage{enumerate}
\usepackage{scalerel,stackengine}
\usepackage{graphics}
\usepackage{fancyhdr}
\usepackage{cancel}
\usepackage{enumerate}
\usepackage{amssymb}
\usepackage{amsmath}
\usepackage{hyperref}
\usepackage{mdwlist}
\usepackage{etoolbox}
\usepackage{latexsym}
\usepackage{amsthm}
\usepackage{multicol}
\usepackage{makeidx}
\usepackage{bm}
\usepackage{dsfont}
\usepackage{graphicx}
\usepackage{wrapfig}
\usepackage{multicol}
\usepackage{makeidx}
\usepackage{bm}
\usepackage{dsfont}
\usepackage{mdframed,color}
\usepackage{amsmath,amssymb}
\usepackage[dvipsnames]{xcolor}
\usepackage{mathtools}

\usepackage[capitalize]{cleveref}
\usepackage{bbm, dsfont}
\usepackage{tikz}
\usetikzlibrary{matrix}
\usetikzlibrary{positioning}

\usepackage{titlesec}
\titleformat{\section}
  {\Large\center\bfseries}
  {\thesection.}{.7em}{}
\titlespacing*{\section}{0pt}{3.5ex plus 0ex minus 0ex}{1.5ex plus 0ex}
\titleformat{\subsection}
  {\center\bfseries}
  {\thesubsection.}{.7em}{}
\titlespacing*{\subsection}{0pt}{3.5ex plus 0ex minus 0ex}{1.5ex plus 0ex}
\titleformat{\subsubsection}
  {\center\bfseries}
  {\thesubsubsection.}{.7em}{}
\titlespacing*{\subsubsection}{0pt}{3.5ex plus 0ex minus 0ex}{1.5ex plus 0ex}

\addto\captionsenglish{}

\usepackage{titling}
\newtheorem{theo}{Theorem}[section]
\newtheorem{cor}[theo]{Corollary}
\newtheorem{conj}[theo]{Conjecture}
\newtheorem{question}[theo]{Question}
\newtheorem{lemma}[theo]{Lemma}
\newtheorem{prop}[theo]{Proposition}
\newtheorem*{prop*}{Proposition}
\newtheorem{thmx}{Theorem}
\newtheorem{corx}[thmx]{Corollary}
\newtheorem*{corxD}{Corollary D}

\newtheoremstyle{definition}{2mm}{2mm}{}{}{\bfseries}{.}{.5em}{}
\theoremstyle{definition}
\newtheorem{defn}[theo]{Definition}
\newtheorem{remark}[theo]{Remark}
\newtheorem{ex}[theo]{Example}

\crefname{theo}{Theorem}{Theorems}
\Crefname{theo}{Theorem}{Theorems}

\crefname{cor}{Corollary}{Corollaries}
\Crefname{cor}{Corollary}{Corollaries}

\crefname{conj}{Conjecture}{Conjectures}
\Crefname{conj}{Conjecture}{Conjectures}

\crefname{question}{Question}{Questions}
\Crefname{question}{Question}{Questions}

\crefname{prop}{Proposition}{Propositions}
\Crefname{prop}{Proposition}{Propositions}

\crefname{thmx}{Theorem}{Theorems}
\Crefname{thmx}{Theorem}{Theorems}

\crefname{corx}{Corollary}{Corollaries}
\Crefname{corx}{Corollary}{Corollaries}

\crefname{defn}{Definition}{Definitions}
\Crefname{defn}{Definition}{Definitions}

\crefname{ex}{Example}{Examples}
\Crefname{ex}{Example}{Examples}

\newcommand{\N}{\mathbb{N}}
\newcommand{\Z}{\mathbb{Z}}
\newcommand{\C}{\mathbb{C}}
\newcommand{\Q}{\mathbb{Q}}
\newcommand{\R}{\mathbb{R}}

\newcommand{\sF}{\mathcal{F}}

\newcommand{\sC}{\mathcal{C}}

\newcommand{\sB}{\mathcal{B}}

\newcommand{\ws}{\textup{\textsf{{WP}}}}

\newcommand{\Rspa}{\R\textup{\textsf{-span}}}
\newcommand{\bE}{\mathop{\mathbb{E}}}

\newcommand{\T}{\mathbb{T}}
\newcommand{\norm}[1]{\left\lVert#1\right\rVert}
\newcommand{\ind}[1]{\mathbbm{1}_{#1}}

\newcommand{\iplim}{\mathop{\textup{\textsf{IP-}}\lim}}

\providecommand{\norm}[1]{\lVert #1\rVert}

\makeatletter
\def\thanks#1{%
  \begingroup
  \let\@thefnmark\relax
  \@footnotetext{#1}%
  \endgroup
}
\makeatother

\begin{document}

\author{Felipe Hernández}
\date{\small \today}
\title{{\bfseries Multiple Polynomial Recurrence in Weyl Systems}}
\maketitle

\begin{abstract}
In this work we give a full characterization of sets of multiple polynomial recurrence in Weyl systems, which are ergodic unipotent affine transformations on products of tori and finite abelian groups. In particular, we show that measurable and topological recurrence in Weyl systems coincide. Our analysis also yields a structure theorem for polynomial multicorrelation sequences in Weyl systems. These results stem from an in-depth study of the Weyl complexity of a set of polynomials and the introduction of a new concept: the \textit{Weyl polynomials} generated by a set of polynomials.
\end{abstract}
\thanks{The author acknowledges OpenAI’s ChatGPT for Academic Researchers program for providing access to ChatGPT during the revision of this article.}

\section{Introduction} 
Recurrence is a prominent area of research in ergodic theory and topological dynamics. One of the central results in this area is Poincaré's recurrence theorem \cite{Poincare1890}, stating that for any measure-preserving system $(X, \mathcal{B}, \mu,T)$ and $A \in \mathcal{B}$ with $\mu(A)>0$, there is $n \in \N$ such that $\mu(A\cap T^{-n}A)>0.$ This foundational result motivated, among many other things, the definition of \textit{sets of measurable recurrence}. A set $R\subseteq \N$ is said to be a \textit{set of measurable recurrence} if for each measure-preserving system $(X, \mathcal{B},  \mu,T)$ and $A\in  \mathcal{B}$ with $\mu(A)>0$, there is $n\in R$ such that $\mu(A\cap T^{-n}A)>0.$ In \cite{Furstenberg77} Furstenberg  proved a far-reaching generalization of Poincaré's recurrence theorem. He showed that for any $k\in \N$, any measure-preserving system $(X, \mathcal{B},  \mu,T)$, and any  $A \in \mathcal{B}$ with $\mu(A)>0$, there is $n \in \N$ such that $\mu(A\cap T^{-n}A\cap \cdots \cap T^{-kn}A)>0$. Furstenberg's result also implies Szemerédi's celebrated theorem on arithmetic progressions \cite{Szemeredi75}.

A polynomial $p\in \Q[x]$ is \textit{integral} if $p(\Z)\subseteq \Z$. A polynomial extension of Furstenberg's multiple recurrence theorem was provided by V. Bergelson and A. Leibman 
 \cite{Bergelson_Leibman96} in their polynomial extension of Szemerédi's theorem. A special case of their result reveals that for any measure-preserving system $(X, \mathcal{B}, \mu,T)$, integral polynomials $p_1(n),\ldots, p_k(n)$ with zero constant term, and $A \in \mathcal{B}$ with $\mu(A)>0$, there is $n\in \N$ such that $\mu(A\cap T^{-p_1(n)}A \cap \cdots \cap T^{-p_k(n)}A)>0$. 

Bergelson and Leibman's result motivates the introduction of more refined recurrence schemes. Given integral polynomials $p_1,\ldots,p_k$ and a measure-preserving system $(X, \mathcal{B},\mu,T)$, a set $R\subseteq \N$ is a set of \textit{$\{p_1(n),\ldots,p_k(n)\}$-measurable recurrence for $(X, \mathcal{B},\mu,T)$} if for each set $A\in \mathcal{B}$ with $\mu(A)>0$ there exist infinitely many $n\in R$ such that $ \mu(A\cap T^{-p_1(n)}A \cap \cdots \cap T^{-p_k(n)}A)>0 $. In addition, we say that a set $R\subseteq \N$ is a set of \textit{$\{p_1(n),\ldots,p_k(n)\}$-measurable recurrence} if it is a set of $\{p_1(n),\ldots,p_k(n)\}$-measurable recurrence for every measure-preserving system. Note that sets of $\{n\}$-measurable recurrence are precisely sets of measurable recurrence. It is worth mentioning that the notion of multiple recurrence for families of integer sequences was introduced and applied to polynomials in \cite[Def. 2.3]{FLWPOWERSOFSEQUENCES} and in \cite[Def. 3.1]{Frantzikinakis11}. 

The notion of $\{p_1(n),\ldots,p_k(n)\}$-measurable recurrence has a natural counterpart in topological dynamics. Given a minimal topological system $(X,T)$, a set $R\subseteq \N$ is a set of \textit{$\{p_1(n),\ldots,p_k(n)\}$-topological recurrence for $(X,T)$} if for any nonempty open set $U\subseteq X$ there exist infinitely many $n\in R$ such that $U\cap T^{-p_1(n)}U\cap \cdots \cap T^{-p_k(n)}U \neq \emptyset$. As in the measurable case, if the system is not specified, we assume the recurrence property holds for all minimal topological systems.

 A natural question arises when considering different recurrence schemes and their interactions. Specifically, given finite sets $P$ and $Q$ of integral polynomials, is it true that for every $R\subseteq \N$, if $R$ is a set of $P$-measurable (respectively topological) recurrence, then it is also a set of $Q$-measurable (respectively topological) recurrence? This type of question was first raised by Bergelson in \cite[Question 7 in Chapter 5]{Bergelson96} where he asked whether for $k\geq 2$ there is a set of $\{n,\ldots,kn\}$-measurable recurrence that is not a set of $\{n,\ldots,(k+1)n\}$-measurable recurrence. For $k=2$ the positive answer to Bergelson's question is contained in \cite[Pages 177-178]{Furstenberg81}. The general result was obtained by Frantzikinakis, Lesigne, and Wierdl \cite{Frantzikinakis_Lesigne_Wierdl_sets_k-recurrence:2006} who showed that $\{n,\ldots, kn \}$-measurable recurrence does not imply $\{n,\ldots,(k+1)n\}$-measurable recurrence. Their argument crucially relies on analyzing recurrence in connected Weyl systems, which are unipotent affine transformations on finite-dimensional tori --- for instance, the two-dimensional torus with the transformation $T(x,y)=(x+\alpha,x+y)$, for $\alpha\in \R\setminus \Q$, is a Weyl system (for a proper definition of Weyl systems, see \cref{sec2}).

Less is known about the interplay of general polynomial recurrence schemes. One of the few results in this direction is due to Bergelson and Håland who  showed in \cite{MR1412598} that the set $\{\lfloor \lfloor \sqrt{2} n\rfloor \sqrt{2}\rfloor : n\in \N\}$ is a set of $\{n^2\}$-measurable recurrence that is not a set of $\{n\}$-measurable recurrence. Later, Frantzikinakis \cite{Frantzikinakis08} showed that for any integral polynomial $p$ with $p(0)=0$ and degree greater than $1$, $\{p(n)\}$-measurable recurrence and $\{n\}$-measurable recurrence are in \textit{general position}. For two sets of integral polynomials $P$ and $Q$, we say that $Q$-measurable recurrence and $P$-measurable recurrence are in general position if there are sets $R_1,R_2\subseteq \N$ such that $R_1$ is a set of $P$-measurable recurrence but not of $Q$-measurable recurrence and $R_2$ is a set of $Q$-measurable recurrence but not of $P$-measurable recurrence. Frantzikinakis' method also crucially relies on the study of Weyl systems. This naturally leads to the following questions.

\begin{question}\label{q4}
 Let $P=\{p_1,\ldots,p_r\}$ and $Q=\{q_1,\ldots,q_l\}$ be two sets of integral polynomials with zero constant term. 
 \begin{enumerate}[{\bf (i)}]
     \item  What are sufficient and necessary conditions on $P$ and $Q$ for the existence of a set $R\subseteq \N$ of $P$-measurable recurrence in Weyl systems that is not a set of $Q$-measurable recurrence in Weyl systems?
     \item   In case it is possible to find conditions providing an answer to question (i), are they sufficient for proving that there is a set $R\subseteq \N$ of $P$-measurable recurrence that is not a set of $Q$-measurable recurrence?
 \end{enumerate}
\end{question}
\cref{q4} (ii) asks if it is sufficient to study Weyl systems to obtain conditions for separating polynomial recurrence schemes. We know that these conditions cannot be necessary. A special case of one of our results (see \cref{characterization-of-recurrence-Weyl} and \cref{ex-double-linear-implies-quadratic}) shows that every set of $\{n,2n\}$-measurable recurrence for Weyl systems is a set of $\{n^2\}$-measurable recurrence for Weyl systems. However, it was shown by Griesmer \cite{GRIESMER_2024} that there is a set of $\{n,2n\}$-measurable recurrence that is not a set of $\{n^2\}$-measurable recurrence, answering a question of Frantzikinakis, Lesigne, and Wierdl \cite[remarks on p.~843]{Frantzikinakis_Lesigne_Wierdl_sets_k-recurrence:2006}.  

A glimpse into a possible answer for \cref{q4} (i) is provided by Host, Kra, and Maass \cite{Host_Kra_Maass_variations_top_recurrence:2016} who proved that a set is a set of $\{n\}$-topological recurrence in Kronecker systems if and only if it is a set of $\{n\}$-topological recurrence in nilsystems. Kronecker systems are ergodic rotations on a compact abelian Lie group, whereas nilsystems correspond to a more general class of systems including Kronecker systems and Weyl systems (see \cref{sec2} for more details and examples). Host, Kra, and Maass' result suggests that it may be possible to characterize schemes of recurrence in Weyl systems through schemes of recurrence in Kronecker systems. Related work was done in \cite[Theorem C]{Huang_Shao_Ye_nilbohr_automorphy:2016} where the family of sets of $\{n,\ldots,kn\}$-recurrence in nilsystems was characterized through special sets in the torus defined with generalized polynomials.

Kronecker systems form the simplest type of dynamical systems, making recurrence easier to understand. For instance, for any set of polynomials $P=\{p_1,\ldots,p_r\}$ it is relatively simple to show that every subset $R\subseteq \N$ is a set of $P$-measurable recurrence for Kronecker systems if and only if $R$ is a set of $P$-topological recurrence for Kronecker systems. While it remains true that every set of $P$-measurable recurrence is a set of $P$-topological recurrence  (see \cref{measure-implies-topological}), the converse is no longer true in every system, as shown in \cite{Kriz87} by Kříž via a combinatorial construction (see also \cite{McCutcheon93,griesmer2023} for details on the dynamical interpretation of Kříž's example). If we consider the possibility that an answer to \cref{q4} involves conditions related to Kronecker systems, then the following question seems natural to consider.

\begin{question}\label{q3}
    Let $P=\{p_1,\ldots,p_r\}$ be a set of integral polynomials with zero constant term. Is it true that every set of $P$-topological recurrence in Weyl systems is a set of $P$-measurable recurrence in Weyl systems?
\end{question}

Our main result gives an answer to \cref{q4} (i) and a positive answer to \cref{q3}. Given polynomials $\{p_1,\ldots,p_r\}\subseteq \Q[x]$ we say that they are \textit{essentially distinct} if, for every $i\ne j$, both $p_i$ and $p_i-p_j$ are non-constant.
  \begin{thmx}\label{teo-A}
    Let $r\in \N$ and $P=\{p_1,...,p_r\}$ be a set of essentially distinct integral polynomials with zero constant term. Then there exist rationally independent integral polynomials $Q=\{q_1,...,q_r\}$ such that for each $R\subseteq \N$, the following are equivalent:
    \begin{enumerate}
        \item  $R$ is a set of $\{p_1,\ldots,p_r\}$-measurable recurrence for Weyl systems,
        \item  $R$ is a set of $\{p_1,\ldots,p_r\}$-topological recurrence for Weyl systems.
        \item $R$ is a set of $\{q_1,\ldots,q_r\}$-recurrence for Kronecker systems. 
    \end{enumerate}
\end{thmx} 
\begin{remark}
    When measurable and topological recurrence coincide—such as in Kronecker systems—we will refer to both simply as ``recurrence.''
\end{remark}
 If we regard a Weyl system as a nilsystem, its phase space $X$ can be realized as the quotient $G/\Gamma$, where $G$ is an $s$-step nilpotent Lie group and $\Gamma$ is a discrete co-compact subgroup of $G$. In this context, we refer to the Weyl system as an \textit{$s$-step Weyl system}, emphasizing its step of nilpotency as a nilsystem. The step of a Weyl system reflects its complexity, which naturally influences recurrence. This is illustrated by a result of Host, Kra, and Maass \cite[Theorem 5.13, Corollaries 5.3 and 5.10]{Host_Kra_Maass_variations_top_recurrence:2016}, which establishes that $l\leq s$ if and only if every set of $(n,\ldots,ln)$-recurrence in $s$-step Weyl systems is a set of $(n,\ldots,ln)$-recurrence in $(s+1)$-step Weyl systems. \cref{teo-B}, which we will presently formulate, provides a substantial improvement of this result, explaining this behavior for any polynomial recurrence scheme in Weyl systems. For this, we make use of \textit{the Weyl complexity} $W(P)$ of a set of integral polynomials $P$, a notion introduced by Bergelson, Leibman, and Lesigne \cite{bergelson:hal-00017730}. The Weyl complexity $W(P)$ is defined as the minimal $k\in \N$ such that for every connected Weyl system $(X,\mu,T)$ the \textit{ergodic averages along $P$} depend only on the $(k-1)$-step nilfactor of $(X,\mu,T)$ (see \cref{sec2} for more details). 

  \begin{thmx}\label{teo-B}
       Let $r\in \N$ and $P=\{p_1,\ldots,p_r\}$ be a set of essentially distinct integral polynomials with zero constant term. Let $s\geq W(P)$. If $R\subseteq \N$ is a set of $P$-recurrence for $s$-step Weyl systems, then it is a set of $P$-recurrence for $(s+1)$-step Weyl systems.
  \end{thmx}
The condition $s\geq W(P)$ in \cref{teo-B} is optimal, in the sense that for $W(P)\geq 2$, there is a set of $P$-recurrence for $(W(P)-1)$-step Weyl systems that is not a set of $P$-recurrence for $W(P)$-step Weyl systems (see \cref{remark-theo-B}).

 A \textit{polynomial multicorrelation sequence} is a sequence $$n\mapsto \int f_0 T^{p_1(n)}f_1\cdots T^{p_r(n)}f_r d\mu,$$ where $p_1,\ldots,p_r$ are essentially distinct integral polynomials, $(X,\mu,T)$ is a measure-preserving system, and $f_0,\ldots,f_r\in L^\infty(X)$.

We refer to the $\R$-vector space generated by the polynomials $\{q_1,\ldots,q_r\}$ appearing in \cref{teo-A} as the \textit{Weyl polynomials} of $P$, and we denote it by $\ws(P)$. The proofs of \cref{teo-A} and \cref{teo-B} rely on the following technical result, which in rough terms states that every polynomial multicorrelation sequence in Weyl systems possesses a ``polynomial Fourier expansion" involving only Weyl polynomials.

 \begin{thmx}[Structure theorem for polynomial multicorrelation sequences in totally ergodic Weyl systems]\label{teo-C}
Let $(X,\mu,T)$ be a totally ergodic Weyl system with Kronecker factor given by the rotation $x\mapsto x+\beta$, where $\beta=(\beta_1,\ldots,\beta_s)\in \T^s$ for some $s\in \N$. Let $P=\{p_1,\ldots,p_r\}$ be a set of essentially distinct integral polynomials with zero constant term. Then for each $f_0,\ldots,f_r\in L^\infty(\mu)$ there are a sequence $(c_l)_{l\in \N} \in \ell^2(\N)$ and integral polynomials $\{q_{l,j}\}_{l\in \N,j\in [s]}\subseteq \ws(P)$ such that
    \begin{equation*}
        \liminf_{L\to \infty} \limsup_{n\to \infty} \left| \int_X f_0T^{p_1(n)}f_1 \cdots T^{p_r(n)}f_r d\mu -\sum_{l=1}^L c_l \cdot e\left(q_{l,1}(n)\beta_1+\cdots+q_{l,s}(n)\beta_s\right)\right |=0.
    \end{equation*}
\end{thmx}
Two complex sequences $(a_n)_{n\in \N}$ and $(b_n)_{n\in \N}$ are \textit{asymptotically uncorrelated} if the following three limits exist in $\C$ and satisfy
\begin{equation*}
    \lim_{N\to \infty} \frac{1}{N}\sum_{n=1}^N a_nb_n=\lim_{N\to \infty}\frac{1}{N}\sum_{n=1}^Na_n  \cdot \lim_{N\to \infty}\frac{1}{N}\sum_{n=1}^Nb_n.
\end{equation*}
As a corollary of \cref{teo-C} we show that the polynomial multicorrelation sequences in totally ergodic Weyl systems for sets of polynomials sharing no nontrivial Weyl polynomial are asymptotically uncorrelated. 
\begin{corx}\label{teo-D}
Let $P=\{p_1,\ldots,p_r\}$ and $Q=\{q_1,\ldots,q_l\}$ be sets of distinct integral polynomials with zero constant term. Then $$\ws(P)\cap \ws(Q)=\{0\}$$ if and only if, for every pair of totally ergodic Weyl systems $(X,\mu,T)$ and $(Y,\nu,S)$, and functions $f_0,\ldots,f_r\in L^\infty(X)$ and $g_0,\ldots,g_l\in L^\infty(Y)$, the sequences
$$ n\mapsto  \int f_0T^{p_1(n)} f_1  \cdots  T^{p_r(n)} f_{r} d\mu, \text{ and } n\mapsto \int g_0S^{q_1(n)} g_1  \cdots  S^{q_l(n)} g_{l} d\nu $$
are asymptotically uncorrelated.
\end{corx}
\begin{remark}
    It follows from our proof that actually in \cref{teo-D} only one of the systems $(X,\mu,T)$ and $(Y,\nu,S)$ needs to be totally ergodic.
\end{remark}
We conjecture that \cref{teo-D} holds for every pair of totally ergodic measure-preserving systems $(X,\mu,T)$ and $(Y,\nu,S)$. This is supported by known results, such as \cite[Proposition 3.2]{Frantzikinakis_Lesigne_Wierdl_sets_k-recurrence:2006} and \cite[Proposition 5.1]{Frantzikinakis08}. This would also provide a positive answer to \cref{q4} (ii).

The paper is organized as follows: In \cref{sec2}, we introduce the basic notions in dynamical systems, ergodic theory, recurrence, and structure theory that will be used throughout the article. In \cref{sec3}, we introduce an equivalent definition of the \textit{Weyl polynomials}. We provide a formula for polynomial multicorrelation sequences in Weyl systems in terms of the Weyl polynomials, obtaining \cref{teo-C} which is later used to prove \cref{teo-D}. In \cref{sec4}, we first establish sufficient conditions for polynomial recurrence in \textit{standard Weyl systems}. Standard Weyl systems are a subclass of Weyl systems satisfying the property that every connected Weyl system is a factor of a finite product of standard Weyl systems (see \cref{sec2} for more details). Using these conditions, we derive necessary and sufficient conditions for measurable and topological polynomial recurrence in $s$-step Weyl systems through polynomial recurrence in Kronecker systems. This provides a comprehensive understanding of how polynomial schemes of recurrence behave in Weyl systems, leading to \cref{teo-A} and \cref{teo-B}.

\subsection{Notation}
We will use the notation $e(x)=e^{2\pi  \mathbf{i} x}$ for $x\in \T$. For $d\in \N$ and a subspace $V\subseteq \R^d$, we denote as $V^\perp$ the orthocomplement of $V$ for the usual Euclidean inner product. For a transformation $T:X\to X$ and a function $f:X\to \R$ we use the notation $Tf=f\circ T$. Given a polynomial $p:\Z\to \Z$ and $k\in \N$ we write
$$p^{[k]}(n) ={ p(n) \choose k},~ \text{ for each } n\in \N.$$

\subsection*{Acknowledgements}
I thank my co-advisor Vitaly Bergelson for introducing me to this topic and for his continued guidance throughout. I am also grateful to my advisor Florian Richter for his numerous insightful comments on this article. We additionally thank Ethan Ackelsberg for many helpful remarks on an earlier draft. Finally, we thank the anonymous reviewer for numerous helpful comments.

\noindent \textbf{Use of technology:} No AI tools were used to produce any of the results. ChatGPT Astra was used for grammar and revision only.

\section{Preliminaries}\label{sec2}
\subsection{Basic notions}\label{basic-notions}
In this paper, we consider a \textit{measure preserving system} $(X,\sB,\mu,T)$ where $(X,\sB,\mu)$ is a probability space that is regular---meaning that $X$ is a compact metric space and $\sB$ is the Borel $\sigma$-algebra associated to $X$---and $T:X\to X$ is an invertible, measurable and measure preserving map. Since $\sB$ always represents the Borel $\sigma$-algebra of $X$, we will often omit it and denote the system simply as $(X,\mu,T)$. A system $(X,\mu,T)$ is said to be \textit{ergodic} if for every $A\in \sB$, $\mu(A\Delta T^{-1}A)=0$ implies $\mu(A)\in \{0,1\}$, and \textit{totally ergodic} if $(X,\mu,T^n)$ is ergodic for each $n\in \N$. For two measure-preserving systems $(X,\mu,T)$ and $(Y,\nu,S)$, we say that $(Y,\nu,S)$ is a factor of $(X,\mu,T)$ if there is a measurable map $\pi:X'\to Y'$ where $X'\subseteq X$ and $Y'\subseteq Y$ are $T$ and $S$ invariant sets with full measure respectively, such that $\pi\circ T(x)= S\circ \pi(x)$ for every $x\in X'$ and $\pi_*\mu=\nu$. If $\pi $ is injective, we say that $X$ and $Y$ are \textit{isomorphic}. 

A pair $(X,T)$ is called a \textit{topological dynamical system} if $X$ is a compact metric space and $T:X\to X$ is a homeomorphism. We say that the system $(X,T)$ is \textit{minimal} if every closed $T$-invariant subset of $X$ is either $X$ or $\emptyset$. For two topological dynamical systems $(X,T)$ and $(Y,S)$, we say that $(Y,S)$ is a (topological) factor of $(X,T)$ if there is an onto continuous map $\pi:X\to Y$ such that $\pi\circ T=S\circ \pi $. If $\pi $ is injective, we say that $X$ and $Y$ are \textit{conjugate}. When we have a measure-preserving system $(X,\mu,T)$ satisfying that $(X,T)$ is a topological dynamical system, we will refer to the triple $(X,\mu,T)$ simply as a \textit{system}, to emphasize that it is regarded simultaneously in both the measure-theoretic and topological settings.

We define several classical families of systems relevant to our work. Following the terminology from Furstenberg \cite{Furstenberg67}, a \textit{Kronecker system} is an ergodic system $(X,\mu,T)$ where $X$ is a compact abelian Lie group, $\mu$ is its Haar measure, and $T$ is given by $T(x)=\tau x$ for every $x\in X$ and some fixed $\tau\in X$. In a Kronecker system \((X, \mu, T)\), the eigenfunctions span \(L^2(X)\) (for a proof of this fact see for instance \cite{Walters82}). A \textit{Weyl system} $(X,\mu,T)$ is an ergodic system where 
$X$ is a compact abelian Lie group (which is the product of a finite dimensional torus and a finite abelian group, see \cite[Proposition 2.42 (i)]{HofmannMorris20}), $\mu$ is the Haar measure on $X$, and $T$ is a unipotent affine transformation, meaning that $T(x)=a\cdot U(x)$, where $a\in X$ and $U$ is a Lie group automorphism of $X$ whose restriction to the connected component of the identity in $X$ (i.e. to the torus component of $X$) is a unipotent automorphism. An \textit{$s$-step nilsystem} is a system $(Y, \mathcal{D}, \nu, S)$ where $Y$ is a homogeneous space $G / \Gamma$ of an $s$-step nilpotent Lie group $G$ and $\Gamma$ is a discrete cocompact subgroup of $G$, $\nu$ is the Haar measure in $Y$, and there is an element $g \in G$ such that $S(x \Gamma)=g x \Gamma$ for each $x \in G$. Let $G_1=G$ and $G_l=\left[G_{l-1}, G\right]$ for $l \geq 2$. Up to isomorphism, one can alternatively define a Weyl system as an ergodic nilsystem $(X=G/\Gamma,\mu,T)$ such that the connected component $G_0$ of the identity $e_G$ in $G$ is abelian (Cf. \cite[Proposition 19 in Chapter 11]{Host_Kra_nilpotent_structures_ergodic_theory:2018}).

Every Kronecker system is a Weyl system, but not the other way around, as the following example shows.
\begin{ex}
Define the system $(\T^2,m_{\T^2},T)$ where $m_{\T^2}$ is the Haar measure on $\T^2$ and $T$ is given by $$T\begin{pmatrix}
    u\\v
\end{pmatrix}=\begin{pmatrix}
    1 &0 \\
    1 & 1
\end{pmatrix} \begin{pmatrix}
    u\\v
\end{pmatrix} + \begin{pmatrix}
    \alpha\\0
\end{pmatrix}, \text{ for }\begin{pmatrix}
    u\\v
\end{pmatrix} \in \T^2$$ where $\alpha\in \R\setminus\Q$. Since one can write $T(x)=Ax+a$ where $a\in \T^2$ and $A$ is a unipotent matrix (as $A-I$ is nilpotent), we have that $(\T^2,m_{\T^2},T)$ is a Weyl system. Moreover, $(\T^2,m_{\T^2},T)$ is a Weyl system that is not a Kronecker system, since the space generated by all the eigenfunctions of $T$ consists of functions $f:\T^2\to \C$ not depending on the second coordinate.
\end{ex}
Likewise, every Weyl system is an ergodic nilsystem. However, not every ergodic nilsystem is isomorphic to a Weyl system.
     \begin{ex}\label{ex-2.2}
     Let
\[
G = \left\{ \begin{pmatrix}
1 & x & z \\
0 & 1 & y \\
0 & 0 & 1
\end{pmatrix} \mid x, y, z \in \R \right\}
\]
be the group of $3\times 3$ upper triangular matrices with $1$s in the diagonal, with the usual matrix multiplication. Then \( (G, \cdot) \) is a Lie group, known as the Heisenberg group. A Heisenberg system is given by the triplet \((X=G / \Gamma, m_X, T)\), where \(\Gamma\) is a discrete subgroup of \(G\) defined by
    \[
    \Gamma = \left\{ \begin{pmatrix}
    1 & \gamma_1 & \gamma_3 \\
    0 & 1 & \gamma_2 \\
    0 & 0 & 1
    \end{pmatrix} \mid \gamma_1, \gamma_2, \gamma_3 \in \Z \right\},
    \]
   \( m_X \) denotes the Haar measure on the quotient space \( G / \Gamma \), and \( T \) represents the transformation induced by left multiplication by the element
   $$\tau =  \begin{pmatrix}
1 & a & 0 \\
0 & 1 & b \\
0 & 0 & 1
\end{pmatrix}$$
where $a,b\in \
\R$ are such that $\{1,a,b\}$ are rationally independent. This system is an example of a nilsystem that is not a Weyl system. Indeed, we follow Remark 0.20 (2) from \cite{Bergelson_Leibman07}. Assume for contradiction that $(X,\mu,T)$ is isomorphic to a Weyl system. This in particular implies that characters span $L^2(X)$. Consider the function $f(x)=e(x_{1,3})$ for $$x=  \begin{pmatrix}
1 & x_{1,2} & x_{1,3}\\
0 & 1 & x_{2,3} \\
0 & 0 & 1
\end{pmatrix}\Gamma\in X,$$ 
where $x_{1,2},x_{1,3},x_{2,3}\in [0,1)$. Let $\chi$ be a character on $X$. One can show that there exists a polynomial $p$ such that for any $n\in \N$, $\chi(\tau^n\Gamma)=e^{2\pi i p(n)}$. On the other hand, notice that 
$$ \tau^n\Gamma=   \begin{pmatrix}
1 & an & \frac{n(n-1)ab}{2} \\
0 & 1 & bn \\
0 & 0 & 1
\end{pmatrix}\Gamma =\begin{pmatrix}
1 & \{an\} & \{\frac{n(n-1)ab}{2}-an[bn]\} \\
0 & 1 & \{bn\} \\
0 & 0 & 1
\end{pmatrix}\Gamma . $$
As the Kronecker factor of $(X,\mu,T) $ corresponds to the ergodic rotation $(x,y)\mapsto (x+a,y+b)$, by \cite[Theorem 11 in Chapter 11]{Host_Kra_nilpotent_structures_ergodic_theory:2018} the transformation $T$ is uniquely ergodic. Consequently, we have that 
$$\int f \overline{\chi} d\mu= \lim_{N\to \infty} \frac{1}{N} \sum_{n=1}^N f(T^n\Gamma) \overline{\chi(T^n\Gamma)}= \lim_{N\to \infty} \frac{1}{N}\sum_{n=1}^N e\left(\frac{n(n-1)ab}{2}-an[bn]-p(n)\right)=0, $$
where - as mentioned in \cite[Remark 0.20 (2)]{Bergelson_Leibman07} - one can check that $$\left(\frac{n(n-1)ab}{2}-an[bn]-p(n)\right)_{n\in \N }$$ is uniformly distributed mod $1$ following the method in \cite[Section 3.6]{Bergelson_Leibman07}. In consequence, $f$ is orthogonal to $\chi$, and as $\chi$ was arbitrary, we conclude that $f$ is orthogonal to all characters of $X$. Since $L^2(X)$ is generated by the characters of $X$ by the Peter-Weyl theorem, this is a contradiction. Therefore, $X$ cannot be isomorphic to a Weyl system.
\end{ex}

The following theorem from Frantzikinakis and Kra characterizes Weyl systems when the associated nilmanifold is connected.
\begin{theo}[{\cite[Cf. Proposition 3.1]{Frantzikinakis_Kra_averages_product_integrals:2005}}]\label{F-K}
    Let $X=G / \Gamma$ be a connected $k$-step nilmanifold such that $G_0$ is abelian. Then every nilrotation $T_a(x)= ax$ defined on $X$ with the Haar measure $m$ is conjugate to a $k$-step unipotent affine transformation on some finite dimensional torus. 
\end{theo}
The next property links connectedness to total ergodicity in ergodic nilsystems.
 \begin{prop}[{\cite[Corollary 7 in Chapter 11]{Host_Kra_nilpotent_structures_ergodic_theory:2018}}]\label{connected}
     For an ergodic nilsystem $(X=G/\Gamma,\mu,T)$, the following are equivalent: 
     \begin{enumerate}
         \item The system is totally ergodic.
         \item The space $X$ is connected.
         \item We have $G=G_0\Gamma$.
     \end{enumerate}
 \end{prop}
We also remark that for nilsystems, the notions of unique ergodicity, ergodicity, minimality, transitivity coincide by \cite[Chapter 11, Theorem 11]{Host_Kra_nilpotent_structures_ergodic_theory:2018}.

When dealing with a connected Weyl system we can always assume that such system is a product of \textit{standard Weyl Systems} by passing through an extension (see \cref{FKLemma}). This name was first defined in \cite{bergelson:hal-00017730}, however it has appeared without this name in other articles such as \cite{Frantzikinakis_Kra_averages_product_integrals:2005}.
\begin{defn}
    We say that $(X=\T^d,m_{\T^d},T)$ is a \textit{standard Weyl System} with rotation by $\alpha\in \R\setminus \Q$ if 
    \begin{equation*}
        T(x_1,\ldots,x_d)=(x_1+\alpha,x_2+x_1,\ldots,x_d+x_{d-1}).
    \end{equation*}
\end{defn}

\begin{lemma}\label{FKLemma}
    Let $T:\T^D \to \T^D$ be defined by $T(x)=Ax+b$, where $A$ is a $D\times D$ unipotent integer matrix satisfying $(A-I_D)^d=0$ and $b\in \T^D$. Assume that $T$ is ergodic. Then, there exist a positive integer $s\leq D$ and $\beta\in \T^s$ such that $T$ is a factor of an ergodic affine transformation $S:\T^{sd} \to \T^{sd}$, defined by 
    \begin{equation}\label{def-of-S}
        Sy= (y_1^T+\beta^T,y_2^T+y_1^T,\ldots,y_d^T+y_{d-1}^T)^T, \text{ for }y=(y_1^T,\ldots,y_d^T)^T\in (\T^s)^d,
    \end{equation}
    where the rotation by $\beta$ on $\T^s$ is isomorphic to the Kronecker factor of $(\T^D,T)$.
\end{lemma}
A related version of \cref{FKLemma} was proved in \cite[Lemma 4.1]{Frantzikinakis_Kra_averages_product_integrals:2005} in a slightly different fashion: the extension $S$ of $T$ acted on the same space $\T^D$ (and not $\T^{sd}$), however the Kronecker factor of $S$ does not necessarily coincide with the Kronecker factor of $T$ in the construction of \cite[Lemma 4.1]{Frantzikinakis_Kra_averages_product_integrals:2005}. On the other hand, in our \cref{FKLemma}, the extension is over a space of potentially higher dimension, but the Kronecker factor is preserved.

Before proving \cref{FKLemma} we need the following lemma.
\begin{lemma}\label{technical-lemma}
Let $T:\T^D \to \T^D$ be defined by $T(x)=Ax+b$, where $A$ is a $D\times D$ unipotent integer matrix satisfying $(A-I_D)^d=0$ and $b\in \T^D$. Assume that $T$ is ergodic. Then, there is a positive integer $s\leq D$ and a matrix $K\in \mathcal{M}_{s\times D}(\Z)$ of full rank such that the map $\pi:\T^D\to \T^s$ given by $\pi(x)=Kx$ is the factor map to the Kronecker factor of $(\T^D,T)$. In addition, $\ker_{\T^D}(K)=\text{Im}_{\T^D}(A-I_D)$ and there is a matrix $C\in \mathcal{M}_{D\times s}(\Z)$ that is a right inverse of $K$, that is, $KC=I_s$ where $I_j$ is the identity matrix of dimensions $j\times j$ for $j\in \N$. 
\end{lemma}
\begin{proof}
    Write $A=I+N$, where $N$ is a $d$-nilpotent matrix. We recall from \cite[Section 5.3]{Host_Kra_Maass_variations_top_recurrence:2016} that $(\T^D,T)$ can be represented as a nilsystem $(X=G/\Gamma,T)$ where $G$ is the Lie group of transformations $\T^D\to \T^D$ generated by $A$ and the rotations by elements of $\T^D$, with composition as operation. Here $\Gamma=\{A^n:n\in \Z\} .$ Thus, $\T^D$ is identified with $G/\Gamma$, where the transformation $T$ is identified with the left rotation by the element $\tau=\tau_b\tau_A$, which is equal to the composition of rotation by $b$ (represented by $\tau_b$) with $A$ (represented by $\tau_A$). In this setting, $G_2\cong \text{Im}_{\T^D}(N)$ and the Kronecker factor is isomorphic to $\T^D/\text{Im}_{\T^D}(N)$. Since $\text{Im}_{\T^D}(N)$ is a closed subgroup of $\T^D$, the quotient $\T^D/\text{Im}_{\T^D}(N)$ is a connected compact abelian Lie group, hence isomorphic to $\T^s$, where $s=D-\text{rank}(N)$. Let $q:\T^D\to\T^D/\text{Im}_{\T^D}(N)$ be the quotient map and choose a continuous
group isomorphism $\theta:\T^D/\text{Im}_{\T^D}(N)\to\T^s$. Then $\pi=\theta\circ q$ is a continuous surjective homomorphism, so $\pi(x)=Kx$ for some integer matrix
$K\in\mathcal{M}_{s\times D}(\Z)$ of rank $s$. Moreover, $\pi$ is the Kronecker factor map and $\ker_{\T^D}(K)=\text{Im}_{\T^D}(N)$ and $KN=0$.

    By Smith normal form, there are unimodular integer matrices $U\in GL_s(\Z)$ and $V\in GL_D(\Z)$ such that $UKV=[M,\mathbf{0}_{s\times (D-s)}]$, where $M=\text{diag}(m_1,\ldots,m_s)$ is a diagonal matrix of size $s\times s$ with $m_1,\ldots,m_s\in \N$, and $\mathbf{0}_{s\times (D-s)}$ is the matrix full of zeros of dimensions $s\times (D-s)$. In particular as $\ker_{\T^D}(K)=\text{Im}_{\T^D}(N)$ which is connected, we have that $\ker_{\T^s}(M)$ is connected. Since $M$ is a diagonal matrix with entries in $\N$, we conclude that $m_i=1$ for each $i \in [s]$, and then $M=I_s$. Setting $C=V [I_s, 0]^T U$, we have that $C$ is an integer matrix satisfying $KC=I_s$, which concludes the proof.
\end{proof}

\begin{proof}[{Proof of \cref{FKLemma}}]
 Write $A=I+N$, where $N$ is a $d$-nilpotent matrix. Let $K$ and $C$ be given by \cref{technical-lemma}. Set $\beta:=Kb\in \T^s$. Since $KC=I_s$, we get that $C$ has rank $s$. Moreover, as $KN=0$, then $\text{Im}_{\T^D}(C)\cap \text{Im}_{\T^D}(N)=\{0\}$. Given that $N$ has rank $D-s$, we have that $\R^D=C \R^s+ N\R^D$. Since $N^d$ is the null matrix, we obtain
    \begin{equation}\label{full-rank}
        \R^D= C \R^s+N\cdot C \R^s + \cdots + N^{d-1}\cdot C \R^s.
    \end{equation}
    This motivates defining the integer matrix $B$ built by blocks:
    $$B= \left[ C, NC,\ldots, N^{d-1}C  \right].$$
    We list some properties of this matrix:
    \begin{enumerate}
        \item\label{I1} By \cref{full-rank}, $B$ has full rank, and thus it defines a surjective morphism from $(\T^s)^d$ onto $\T^D$.
        \item\label{I3} If $J$ denotes the matrix defined by \cref{def-of-S}, i.e. $\text{ for }y=(y_1^T,\ldots,y_d^T)^T\in (\R^s)^d$: 
        \begin{equation}\label{def-J}
           J(y_1^T,\ldots,y_d^T)^T=(y_1^T,y_2^T+y_1^T,\ldots,y_d^T+y_{d-1}^T)^T, 
        \end{equation}
        then $AB=BJ$. Indeed, on one hand, 
        $$(A-I)B=NB= \left[NC,N^2C,\ldots,N^{d-1}C,\mathbf{0} \right]. $$
        On the other hand, observing that by \cref{def-J}, the matrix $(J-I)$ shifts to the right coordinates when applied to a vector in $(\R^{s})^d$, we get that for $y=(y_1^T,\ldots,y_{d}^T)^T\in \R^{sd}$
        \begin{align*}
            B(J-I)y&=[C,NC,\ldots,N^{d-1}C](0,y_1^T,\ldots,y_{d-1}^T)^T\\
            &= NCy_1+\cdots+ N^{d-1}C y_{d-1}= \left[NC,\ldots, N^{d-1}C,\mathbf{0} \right]y,
        \end{align*}
        concluding that $AB=BJ$.
    \end{enumerate}
    From the equality $Kb=\beta=KC\beta$ in $\T^s$, and the fact $\ker_{\T^D}(K)=\text{Im}_{\T^D}(N)$, we get that there is $a\in \T^D$ such that $Na=(C\beta-b)$. 
    Define the map $\psi:\T^{sd}\to \T^D$ given by $\psi(y)=By+a$. Then $\psi$ is a factor map between $(\T^{sd},S)$ and $(\T^D,T)$. Indeed, the fact that $B$ has integer entries implies that $\psi$ is well defined. By \cref{I1} we have that $\psi$ is surjective. Set $\beta'=(\beta^T,0,\ldots,0)^T\in \T^{sd}$. Then, for every $y\in \T^{sd}$,
    $$\psi(Sy)= \psi(Jy+ \beta')= BJy+ B\beta' +a= ABy+ C\beta +a = ABy +(I+N)a +b = T\psi(y),   $$
    where we used \cref{I3} in the third equality, concluding that $\psi$ is a factor. Finally, since $(\T^D,T)$ is ergodic, we have that the coordinates of $\beta$ and $1$ are rationally independent. Since the Kronecker factor of $(\T^{sd},S)$ coincides with the one of $(\T^D,T)$ (as it is given by $\T^{sd}/\text{Im}_{\T^{sd}}(J-I)$, which only preserves the first set of coordinates $\T^s$) we deduce that $(\T^{sd},S)$ is ergodic, concluding the lemma.
\end{proof}

To conclude this section, we show that every set of $P$-measurable recurrence is a set of $P$-topological recurrence. This result is well known for $P=\{n\}$, but since the proof remains valid for any set of integral polynomials $P$, we include it here for completeness. 
\begin{prop}\label{measure-implies-topological}
    Let $P=\{p_1,\ldots,p_d\}$ be a set of integral polynomials. Let $(X,T)$ be a minimal topological system. Then, for every $T$-invariant measure $\mu\in M(T)$ we have that every set $R\subseteq \N$ of $P$-measurable recurrence for $(X,\mu,T)$ is a set of $P$-topological recurrence for $(X,T)$.
\end{prop}
\begin{proof}
Since $(X,T)$ is minimal, every invariant measure $\mu$ has full support. That being so, if $R\subseteq \N$ is a set of $P$-measurable recurrence for $(X,\mu,T)$, then for each nonempty open set $U\subseteq X$ there are infinitely many $n\in R$ such that $\mu(U\cap T^{-p_1(n)}U\cap \cdots \cap T^{-p_d(n)}U)>0. $ This in particular implies $U\cap T^{-p_1(n)}U\cap \cdots \cap T^{-p_d(n)}U\neq \emptyset, $ concluding that $R$ is a set of $P$-topological recurrence for $(X,T)$.
\end{proof}
\subsection{Characteristic factors and complexities}
For our study of recurrence, it will be crucial to characterize recurrence in a system through recurrence in a simpler factor. To achieve this, it is crucial to understand the so-called \textit{characteristic factors} of a system, which correspond to the factors that control ergodic averages for polynomial patterns.
\begin{defn}
    Let $(X,\mu,T)$ be an invertible measure-preserving system and $(Y,\nu,T)$ a factor of $(X,\mu,T)$. Let $P=\{p_1,\ldots,p_r\}$ be a set of essentially distinct integral polynomials. We say that \textit{$(Y,\nu,T)$ is characteristic for the ergodic averages over the pattern $P$} if for any functions $f_1, \ldots, f_r \in L^{\infty}(X)$ such that there is $i\in \{1,\ldots,r\}$ with $\bE(f_i| Y)=0$, 
\begin{align*}
\lim _{N-M \rightarrow \infty} \norm{\frac{1}{N-M} \sum_{n=M+1}^N  T^{p_1(n)}f_1 \cdots T^{p_r(n)}f_r }_{L^2(X)}=0.
\end{align*} 
\end{defn}
The following result describes the characteristic factor for the ergodic averages over a polynomial pattern.
\begin{theo}[{\cite[Host \& Kra]{Host_Kra05b}, \cite[Leibman]{Leibman05}}]\label{Finiteness-of-HK}
    Let $P=\{p_1,\ldots,p_r\}$ be a set of essentially distinct integral polynomials. Then, there is $d\in \N$ such that for every invertible ergodic system some characteristic factor for the averages over $P$ is an inverse limit of $d$-step nilsystems.
\end{theo}

\begin{defn}
    Let $P=\{p_1,\ldots,p_r\}$ be essentially distinct integral polynomials. We define the \textit{Host-Kra complexity} $HK(P)$ of $P$ as the minimal $k$ such that for every invertible totally ergodic measure-preserving system $(X,\mu,T)$, there is an inverse limit of $(k-1)$-step nilsystems that is characteristic for the ergodic averages over the pattern $P$.
\end{defn}

A related notion of complexity was introduced by Bergelson, Leibman, and Lesigne \cite{bergelson:hal-00017730}. This complexity is particularly relevant for our work, as it identifies the smallest characteristic factor within the family of Weyl systems.
    \begin{defn}
    Let $P=\{p_1,\ldots,p_r\}$ be essentially distinct integral polynomials. We define the \textit{Weyl complexity} $W(P)$ as the minimal $k$ such that for every connected Weyl system $(X,\mu,T)$, its maximal $(k-1)$-step Weyl factor $(Y,\nu,T)$ is characteristic for the ergodic averages over the pattern $P$.
\end{defn}
\begin{remark}
    It is possible to define Host-Kra complexity and Weyl complexity for each polynomial $p_i$ in $P=\{p_1,\ldots,p_r\}$ as is done in \cite{KUCA_2023} for example. Nevertheless, we avoid doing so as we do not need such precision.
\end{remark}

\section{The Weyl Polynomials}\label{sec3}
\subsection{Definition and properties}
We will need the following notation borrowed from \cite{bergelson:hal-00017730}.
\begin{defn}
For a matrix function $M:\R\to \R^{d\times k}$ we define in $\R^d$ the span of the matrix $M$ as the span of the columns, namely if $M=(M_{1},\ldots,M_k)$ where $M_i$ is the $i$-th column of $M$, then
$$\Rspa(M) = \text{span}_\R\left\{\sum_{i=1}^k \alpha_iM_i(x) \in \R^d \mid x\in \R,  \alpha_1,\ldots,\alpha_k\in \R\right\}. $$
\end{defn}
\begin{remark}
     For a matrix $M:\Z\to \Z^{d\times k}$ with each coordinate $M_{i,j}$ of $M$ being a polynomial, when we write $\Rspa(M)$ we assume implicitly that we mean the span of the unique polynomial map $\tilde{M}:\R\to \R^{d\times k}$ such that the restriction of $\tilde{M}$ to $\Z$ is $M$.
\end{remark}
 While the Weyl complexity provides essential information about the characteristic factors for studying ergodic averages for a pattern given by a set of polynomials $P$, it does not give further insights about the limit itself. However, for our purposes we require a deeper understanding. To be more precise, we need a description on how the polynomial multicorrelation sequences 
  \begin{equation}\label{correlation-sequences}
      n\mapsto \int f_0 T^{p_1(n)}f_1 \cdots T^{p_r(n)}f_r d\mu
  \end{equation}
  evolve in Weyl systems. This is particularly relevant in the 
 case $f_0=f_1=\cdots =f_r=\ind{A}$, where the sequence in \cref{correlation-sequences} becomes $n\mapsto \mu(A\cap T^{-p_1(n)} A \cap \cdots \cap T^{-p_r(n)}A).$ In order to analyze this finer structure, we introduce a new object of study, the \textit{Weyl polynomials}. Before doing so, we first need some preliminary definitions. Recall that for a polynomial $p:\Z\to \Z$ and $k\in \N$ we denote
$$p^{[k]}(n)=\binom{p(n)}{k}, ~ \text{for }n\in \N. $$
\begin{defn}
    Let $P=\{p_1,...,p_r\}$ be a set of integral polynomials. We define for $k\in \N$ the matrix with polynomial entries given by 
    \begin{align*}
    \Lambda_k(P)& =\left(\begin{array}{cccccc}p_1 & 0 & 0 & \ldots & 0 & 0 \\ \vdots & \vdots & \vdots & & \vdots & \vdots \\ p_r & 0 & 0 & \ldots & 0 & 0 \\ p_1^{[2]} & p_1 & 0 & \ldots & 0 & 0 \\ \vdots & \vdots & \vdots & & \vdots & \vdots \\ p_r^{[2]} & p_r & 0 & \ldots & 0 & 0 \\ \vdots & \vdots & \vdots & & \vdots & \vdots \\ \vdots & \vdots & \vdots & & \vdots & \vdots \\ p_1^{[k-1]} & p_1^{[k-2]} & p_1^{[k-3]} & \ldots & p_1 & 0 \\ \vdots & \vdots & \vdots & & \vdots & \vdots \\ p_r^{[k-1]} & p_r^{[k-2]} & p_r^{[k-3]} & \ldots & p_r & 0 \\ p_1^{[k]} & p_1^{[k-1]} & p_1^{[k-2]} & \ldots & p_1^{[2]} & p_1 \\ \vdots & \vdots & \vdots & & \vdots &\vdots \\ p_r^{[k]} & p_r^{[k-1]} & p_r^{[k-2]} & \ldots & p_r^{[2]} & p_r\end{array}\right).
        \end{align*}
        If the set of polynomials $P$ is clear for the context, we will just write $\Lambda_k$ (or $\Lambda_k(n)$ to evaluate the matrix in $n\in \N$). 
\end{defn}
The matrix $\Lambda_d(P)$ appears naturally when describing the polynomial orbits of points in a standard Weyl system. Indeed, following \cite[Section 4]{bergelson:hal-00017730}, for a standard Weyl system $(X=\T^d,\mu,T)$ with rotation by $\alpha\in \R\setminus\Q$, we have that 
$$         T^n(x_1,\ldots,x_d)=\left(x_1+\binom{n}{1}\alpha,x_2+\binom{n}{1}x_1+ \binom{n}{2}\alpha,\ldots,x_d+\sum_{i=1}^{d-1} \binom{n}{i} x_{d-i}+ \binom{n}{d}\alpha \right).$$
In consequence, for an integral polynomial $p:\R\to \R$, we get that
\begin{align*}
   &T^{p(n)}(x_1,\ldots,x_d) - (x_1,\ldots,x_d) =\\
   &\left(p(n)^{[1]}\alpha,p(n)^{[1]}x_1+ p(n)^{[2]}\alpha,\ldots,\sum_{i=1}^{d-1} p(n)^{[i]} x_{d-i}+ p(n)^{[d]}\alpha \right)\\
   &=  \alpha\Lambda_d(\{p\})_1 +  x_1\Lambda_d(\{p\})_2 +\cdots+ x_{d-1}\Lambda_d(\{p\})_d,
\end{align*}
where for each $i\in [d]$, $\Lambda_d(\{p\})_i$ is the $i$-th column of $\Lambda_d(\{p\})$. As a result, for a point $x$ in $[0,1)^d$ such that $\{1,x_1,\ldots,x_d,\alpha\}$ are rationally independent over $\Q$, and integral polynomials $p_1,\ldots,p_r$ with zero constant term, we have that after reordering coordinates accordingly
    \begin{equation}\label{polynomial-orbits}
     \overline{\{ (T^{p_1(n)}x - x,\ldots, T^{p_r(n)}x-x)^T : n\in \Z \}}=\Rspa(\Lambda_d) \text{ mod } 1.
    \end{equation}
     
Moreover, these matrices also determine the Weyl complexity of $P$.   
    \begin{prop}[{\cite[Proposition 4.1]{bergelson:hal-00017730}}]\label{Second-def-W(P)}
        Let $P=\{p_1,\ldots,p_r\}$ be a set of essentially distinct integral polynomials with zero constant term. The Weyl complexity $W(P)$ equals the minimal $k$ for which
        \begin{equation}\label{Weyl-complexity-condition}
            \dim{\Rspa(\Lambda_k(P))}= \dim{\Rspa(\Lambda_{k-1}(P))}+r,
        \end{equation}
        where we set the convention $\dim{\Rspa(\Lambda_{0}(P))}=0$.
    \end{prop}
\begin{remark}\label{remark-after-weyl-cc}
    It follows from the discussion following \cite[Equation (4.3)]{bergelson:hal-00017730} that \cref{Weyl-complexity-condition} holds for every $k\geq W(P)$.
\end{remark}
    
We recall that for $d\in \N$ and a subspace $V\subseteq \R^d$, the orthocomplement $V^\perp$ of $V$ is taken with respect to the Euclidean inner product in $\R^{d}$.
    
Notice that we can realize $\Rspa(\Lambda_{k-1})$ as a subspace of $\Rspa(\Lambda_{k}) $ via the identification $\Rspa(\Lambda_{k-1}) \cong \{0\}^r\times \Rspa(\Lambda_{k-1})$ where the latter is contained in $\Rspa(\Lambda_{k})$. Thus, \cref{Second-def-W(P)} implies that if the step of the standard Weyl system is at least $W(P)$, then for $\mu$-a.e. $x\in X$ the polynomial orbit in \cref{polynomial-orbits} has reached a point of ``oversaturation,'' meaning that all the information about the orbit is encoded in its projection to the $W(P)$-Weyl factor of the system. 

On the other hand, the condition in \cref{Weyl-complexity-condition} amounts to 
$$ \dim{\Rspa(\Lambda_{k})\cap (\R^r\times \Rspa(\Lambda_{k-1} )^\perp)}= r.$$
 The advantage of this equivalent form of \cref{Second-def-W(P)} is that the elements in $\Rspa(\Lambda_{k})\cap (\R^r\times \Rspa(\Lambda_{k-1})^\perp )$ encapsulate the polynomial frequencies that $k$-step Weyl systems can generate but $(k-1)$-step Weyl systems cannot. To be more precise, if $\mathbf{e_i}$ denotes the $i$-th canonical vector of $\R^k$ for each $i\in \{1,\ldots,k\}$, then
$$\Lambda_k = \Lambda_k [\mathbf{e_1}, \mathbf{e_2},\ldots, \mathbf{e_k}]= [\Lambda_k\mathbf{e_1}, \Lambda_k\mathbf{e_2},\ldots, \Lambda_k\mathbf{e_k}]= [\Lambda_k\mathbf{e_1}, \left(\begin{array}{c} \mathbf{0} \\\Lambda_{k-1} \end{array} \right)]  $$
where $\mathbf{0}$ is the matrix of dimensions $r\times (k-1)$ full with zeros, yielding
\begin{equation}\label{eq-orthocomplement-rspan}
   \Rspa(\Lambda_{k})^\perp= \Rspa(\Lambda_k\mathbf{e_1})^\perp \cap (\R^r\times \Rspa(\Lambda_{k-1})^\perp ). 
\end{equation}
This implies that the function 
\[
\xi_k : \R^{r\cdot k} \to \R[n],
\]
\[
\mathbf{v} \mapsto \xi_k(\mathbf{v}) =\mathbf{v}^T \Lambda_k\mathbf{e_1}
\]
when restricted to $\Rspa(\Lambda_k) \cap \big(\R^r \times (\Rspa(\Lambda_{k-1})^\perp)\big)$ is injective. Indeed, for $\mathbf{v}\in  \Rspa(\Lambda_k) \cap (\R^r\times \Rspa(\Lambda_{k-1} )^{\perp})$ if $\xi_k(\mathbf{v})=0$ then 
\begin{align*}
    \mathbf{v}\in &\left( \Rspa(\Lambda_k\mathbf{e_1})^\perp \right) \cap \left( \Rspa(\Lambda_{k} )\cap (\R^r\times \Rspa(\Lambda_{k-1} )^{\perp})\right)\\&= \Rspa(\Lambda_k)^\perp \cap \Rspa(\Lambda_{k} )= \{0\},
\end{align*}
where in the first equality we used \cref{eq-orthocomplement-rspan}. Consequently, the space $ \Rspa(\Lambda_k) \cap \big(\R^r \times (\Rspa(\Lambda_{k-1})^\perp)\big)$ is in bijection with the image of this restriction of $\xi_k$, which is a space of polynomials. These polynomials not only capture the gap in dimension between $\Rspa(\Lambda_{k}) $ and $\Rspa(\Lambda_{k-1}) $, but also define important frequencies that appear in $k$-step Weyl systems. This motivates the following definition.
\begin{defn}[Weyl Polynomials]
Let $P=\{p_1,...,p_r\}$ be a set of essentially distinct integral polynomials. We define the set of \textit{$k$-Weyl polynomials generated by $P$} as
\begin{equation*}
   \ws_0(P):= \{0\},\quad \ws_1(P):=  \xi_1\left(\R^r\right)=\langle P\rangle .
\end{equation*}
and 
\begin{equation*}
    \ws_k(P):=  \xi_k\left(\R^r\times \Rspa(\Lambda_{k-1}(P) )^{\perp}\right),~ \text{if } k\geq 2.
\end{equation*}
 When $k=W(P)$, we will just denote $ \ws(P)$ and call it the \textit{Weyl polynomials} generated by the set $P$.
\end{defn}
\begin{remark}
   In the definition of the Weyl polynomials, we omit the restriction $\mathbf{v}\in \Rspa(\Lambda_k) $ as for each $\mathbf{v}\in \Rspa(\Lambda_k(P))^\perp$ we have $\xi_k(\mathbf{v})=0$.
\end{remark}
\begin{remark}
       For $ \mathbf{v}\in \R^r\times \Rspa(\Lambda_{k-1} )^{\perp}$ the only coordinate of the product $\mathbf{v}^T\cdot \Lambda_k(P)$ that is possibly nonzero is the first one. 
\end{remark}

    We notice that for each $k\in \N$ and for each set of essentially distinct integral polynomials $P=\{p_1,\ldots,p_r\}$, $\ws_k(P)$ is an $\R$-vector space. In particular, when $k=W(P)$, the space $\ws(P)$ has dimension $r$, as it is isomorphic to $ \Rspa(\Lambda_k(P)) \cap \big(\R^r \times (\Rspa(\Lambda_{k-1}(P))^\perp)\big)$. Moreover, as the polynomials $P$ are integral, it is possible to find a basis $Q=\{q_1,\ldots,q_r\}$ of $r$ integral polynomials in $\ws(P)$.

The following proposition is a corollary of \cref{Second-def-W(P)} and can be seen as an alternative definition to Weyl complexity.
\begin{prop}\label{Weyl-complexity-and-space}
Let $P = \{p_1,\dots,p_r\}$ be a set of essentially distinct
integral polynomials with zero constant term. Then the sequence $(\ws_i(P))_{i\ge 0}$ is nondecreasing:
\[
  \ws_k(P) \subseteq \ws_i(P) \quad \text{whenever } k \le i,
\]
and it stabilizes at the index $W(P)$ in the sense that
\[
  \ws_{W(P)-1}(P) \subsetneq \ws_{W(P)}(P)
  \quad\text{and}\quad
  \ws_i(P) = \ws_{W(P)}(P)
  \quad \text{for all } i \ge W(P).
\]
\end{prop}
\begin{proof}
First, we show that the sequence is non-decreasing. Let $k\in \N$ and let $q\in \ws_k(P)$. Let $\mathbf{v}\in \R^{kr}$ such that $\mathbf{v}^T\Lambda_{k}(n)= (q(n),0,\ldots,0)$. Define $w=(\mathbf{v},\Vec{0})\in \R^{(k+1)r}$. We observe that $w^T\Lambda_{k+1}(n)=(\mathbf{v}^T\Lambda_{k}(n),0)=(q(n),0,\cdots,0)$ and thus $q\in \ws_{k+1}(P)$. Therefore, we conclude that $\ws_{k}(P)\subseteq \ws_{k+1}(P)$.

Second, we see that the sequence is constant from $k=W(P)$. Let $k>W(P)$, $q\in \ws_k(P)$, and $\mathbf{v}\in \R^{kr}$ such that
\begin{equation}\label{eq-12}
   \mathbf{v}^T\Lambda_{k}(n)= (q(n),0,\ldots,0) .
\end{equation}
 This implies that the vector $(v_{r+1},\ldots,v_{kr})^T$ is orthogonal to $\mathrm{R}\text{-span }\Lambda_{k-1}$. Since $k-1\geq W(P)$, we have that $v_{i}=0$ for all $i>W(P)\cdot r$ thanks to \cref{Second-def-W(P)}. In particular 
 \begin{equation*}
   (q(n),0,\ldots,0) =\mathbf{v}^T\Lambda_{k}(n)= \Big(  (v_1,\ldots,v_{W(P)\cdot r})  \Lambda_{W(P)}(n),0,\ldots,0\Big),
\end{equation*}
which implies that $q\in \ws_{W(P)}(P)$.

Finally, we see that for $k<W(P)$, we have $\ws_k(P)\neq \ws_{W(P)}(P)$. Indeed, given that $k<W(P)$, by \cref{Second-def-W(P)} and \cref{remark-after-weyl-cc} there is $\mathbf{w}\in  \Rspa( \Lambda_{k})^{\perp}$ such that not all its last $r$ coordinates are $0$. We choose one of such $\mathbf{w}$ with maximal $j\in \N_0$ satisfying that there exists $\mathbf{v}\in \R^{r\cdot j}$ such that $(\mathbf{v},\mathbf{w})\in \Rspa(\Lambda_{k+j})^{\perp}$ (we know that such $j$ exists by finiteness of the Weyl complexity, which follows from \cref{Finiteness-of-HK}). 
Define $q\in \ws_{k+j+1}(P)$ through $(0,\ldots,0,\mathbf{v}^T,\mathbf{w}^T)\Lambda_{k+j+1}=(q,0,\ldots,0)$. If $q\in \ws_{k}(P)$, then there should be $\mathbf{w}'\in \Rspa(\Lambda_{k-1})^\perp $ and $\mathbf{v}'\in \R^r$ such that $(\mathbf{v}'^T,\mathbf{w}'^T)\Lambda_{k}=(q,0,\ldots,0)$. Then we will have that the vector $-(\mathbf{v}',\mathbf{w}',0,\cdots,0)+(0,\ldots,0,\mathbf{v},\mathbf{w})$ is in $\Rspa(\Lambda_{k+j+1}(P))^\perp$ and has its last $r$ coordinates non-zero, which contradicts the choice of $\mathbf{w}$ (i.e., the maximality of $j$), which implies that $q\notin \ws_k(P)$, concluding.
\end{proof}
   For $k\leq W(P)$ it is possible that $\ws_{k-1}(P)=\ws_k(P)$, as the following example shows. 
\begin{ex}
   Consider $P=\{n,2n,n^2\}$. It can be shown that $W(P)=3$ and that $\ws_1(P)=\ws_2(P)$ is generated by $\{n,n^2\}$. Meanwhile, $\ws_3(P)$ is generated by $\{n,n^2,n^3-\frac{n^4}{2}\}.$
\end{ex}
We do not know exactly when $\ws_{k-1}(P)\neq \ws_{k}(P)$. A better understanding of this behavior could provide an answer for a conjecture of Bergelson, Leibman, and Lesigne \cite{bergelson:hal-00017730}, stating that $W(P)\leq |P|$.

We now list some additional useful properties of the space of Weyl polynomials
\begin{prop}\label{general-props}
   Let $P$ and $Q$ be sets of essentially distinct integral polynomials with zero constant term and $d\in \N$. Then
   \begin{enumerate}
          \item The set $P$ is contained in $\ws_d(P)$.
        \item If $W(P)=1$ then $P$ is a basis of $\ws_d(P)$.
       \item If $P\subseteq Q$ then $\ws_d(P)\subseteq \ws_d(Q)$.
       \item If $q$ is a nonconstant integral polynomial with zero constant term, then $\ws_d(P\circ q)= \ws_d(P)\circ q$.
       \item If all polynomials in $P$ are linear, then $\ws(P)=\text{span}_\R\{n,\ldots,n^{|P|}\}$.
   \end{enumerate}
\end{prop}
\begin{proof} 
For proving $\textbf{1.}$, let $r=|P|$. For $i\in [r]$, if we take $\mathbf{v}$ as the vector that has $0$ in all coordinates but on the $i$-th coordinate, in which $v_i=1$ then we have that $\mathbf{v}^T \Lambda_{d}(P)=(p_i,0,\ldots,0)$, and thus $p_i\in \ws_d(P)$.

We note that $\textbf{2.}$ follows from \cref{Weyl-complexity-and-space} and the fact that $\ws_1(P)=\langle P \rangle$.

For proving $\textbf{3.}$, if $P\subseteq Q$, then $\Lambda_{d}(P)$ can be obtained as a submatrix of $\Lambda_{d}(Q)$ (eliminating the rows associated to $Q\setminus P$). Hence, if $\mathbf{v}\in \R^{|P|\cdot d }$ is such that $\mathbf{v}^{T}\Lambda_{d}(P)=(q,0,\ldots,0)$ with $q\in \ws_d(P)$, then by extending $\mathbf{v}$ with $0$'s to $\mathbf{v}'$ we have that $\mathbf{v}'^{T}\Lambda_{d}(Q)=(q,0,\ldots,0)$, getting that $q\in \ws_d(Q)$.

In the case of $\textbf{4.}$, for $\mathbf{v}\in \R^r\times \Rspa(\Lambda_{d-1}(P) )^{\perp}$ we have that 
$$\mathbf{v}^t \Lambda_{d}(P\circ q)(n) \mathbf{e_1}=\mathbf{v}^t \Lambda_{d}(P)(q(n)) \mathbf{e_1}, $$
which implies that $\ws_{d}(P\circ q)= \ws_{d}(P)\circ q$.

Finally, we prove $\textbf{5.}$. We already know from \cite{bergelson:hal-00017730} that if $P$ consists of $r$ different linear polynomials, then $W(P)=r$. We prove the statement then by induction over $r$. The case $r=1$ is obvious. Suppose that the statement is true for $r\geq 1$. Let $P=\{p_1,\ldots,p_{r+1}\}$ be a set of linear polynomials with zero constant term. The set $\ws(P)$ contains all polynomials of degree $r$ with zero constant term, given that it contains $\ws(P\setminus\{p_{r+1}\})$ by $\textbf{3.}$. On the other hand, if $\ws(P)$ does not contain a polynomial of degree $r+1$, then we have that $\ws(P)=\ws_{r+1-1}(P)$ which contradicts \cref{Weyl-complexity-and-space}. We conclude that $\ws(P)$ contains a polynomial of degree $r+1$ which implies that it contains all polynomials of degree $r+1$ with zero constant term, given that $\ws(P)$ is an $\R$-vector space.
\end{proof}

To finish this section, we provide an alternative way of defining the space of Weyl polynomials, replacing the powers $p^{[i]}$ in the matrix $\Lambda_k(P)$ by the powers $p^{i}/i!$.
\begin{defn}
    Let $P=\{p_1,...,p_r\}$ be a set of integral polynomials. We define for $k\in \N$ the matrix with polynomial entries given by 
    \begin{align*}
    \tilde{\Lambda}_k(P)& =\left(\begin{array}{cccccc}p_1 & 0 & 0 & \ldots & 0 & 0 \\ \vdots & \vdots & \vdots & & \vdots & \vdots \\ p_r & 0 & 0 & \ldots & 0 & 0 \\ p_1^{2}/2! & p_1 & 0 & \ldots & 0 & 0 \\ \vdots & \vdots & \vdots & & \vdots & \vdots \\ p_r^{2}/2! & p_r & 0 & \ldots & 0 & 0 \\ \vdots & \vdots & \vdots & & \vdots & \vdots \\ \vdots & \vdots & \vdots & & \vdots & \vdots \\  p_1^{k-1}/(k-1)! & p_1^{k-2}/(k-2)! & p_1^{k-3}/(k-3)! & \ldots & p_1 & 0 \\ \vdots & \vdots & \vdots & & \vdots & \vdots \\ p_r^{ k-1 }/(k-1)! & p_r^{k-2}/(k-2)! & p_r^{ k-3 }/(k-3)! & \ldots & p_r & 0 \\ p_1^{k}/k! & p_1^{k-1}/(k-1)! & p_1^{k-2}/(k-2)! & \ldots & p_1^{ 2 }/2! & p_1 \\ \vdots & \vdots & \vdots & & \vdots &\vdots \\ p_r^{k}/k! & p_r^{ k-1 }/(k-1)! & p_r^{ k-2 }/(k-2)! & \ldots & p_r^{ 2 }/2! & p_r\end{array}\right).
        \end{align*}
\end{defn}
\begin{prop}\label{change-of-basis}
    Let $P=\{p_1,...,p_r\}$ be a set of essentially distinct integral polynomials. Let $k\in \N$. Define the function
    \[
\tilde{\xi}_k : \R^{r\cdot k} \to \R[n],
\]
\[
\mathbf{v} \mapsto \tilde{\xi}_k(\mathbf{v}) :=\mathbf{v}^T \tilde{\Lambda}_k(P)\mathbf{e_1}.
\]    
We have that for $k\geq 2$,
\begin{equation*}
    \ws_k(P)=  \tilde{\xi}_k\left(\R^r\times \Rspa(\tilde{\Lambda}_{k-1}(P) )^{\perp}\right).
\end{equation*}
\end{prop}
\begin{proof}
    The idea of the proof is to do a change of polynomial basis, from $n^i/i!$ to $n^{[i]}$ and vice versa. Define the coefficients $\{c_{i,j}\}_{i,j\in \{0,\ldots,k\}}$ by the change of coordinates
    $$ n^{[i]}= \sum_{j=0}^i c_{i,j} \left(\frac{n^j}{j!} \right)  $$
    for each $i\in\{0,\ldots,k\}$. One observation, that we will use subsequently, is that due to Vandermonde's identity, we have that for $l\leq k$ and for each $i,j\geq 0$,
    \begin{equation}\label{VdM-ID}
       \sum_{h=0}^l c_{h,i} c_{l-h,j}=c_{l,i+j} 
    \end{equation}
    
    Define the lower triangular matrices
    $$A_k:=\left( c_{i,j} \right)_{i,j\in \{1,\ldots,k\}}, \text{ and } B_k:= \left( c_{i,j} \right)_{i,j\in \{0,\ldots,k-1\}}. $$
    We observe that $A_k$ and $B_k$ have entries in $\Q$, and that they have only $1$'s in the diagonal. Thus, we conclude that $A_k$ and $B_k$ are invertible, with inverse having rational entries. We claim that
    \begin{equation}\label{lambdas-equality}
        \Lambda_k(P)B_k = \left(A_k \otimes I_r \right)  \tilde{\Lambda}_{k}(P)
    \end{equation}
    for each $k\geq 2$, where $I_r$ is the identity of dimensions $r\times r$ and $(A_k \otimes I_r)$ represents the $(k\cdot r)\times (k\cdot r)$ block matrix obtained by replacing each entry $c_{i,j}$ by $c_{i,j}I_r$. Indeed, for $1\leq j\leq l\leq k$, on one hand we have
\[
    \bigl(\Lambda_k(P)B_k\bigr)_{(l-1)r+i,j}
    =
    \sum_{h=j}^{l}
    \binom{p_i}{l-h+1}c_{h-1,j-1}.
\]
On the other hand:
\[
    \bigl(\left(A_k \otimes I_r \right) \widetilde{\Lambda}_k(P)\bigr)_{(l-1)r+i,j}
    =
    \sum_{h=j}^{l}
    c_{l,h}\frac{p_i^{h-j+1}}{(h-j+1)!}.
\]
Notice that coordinatewise, the equality in \cref{lambdas-equality} holds if $j>l$, as both entries vanish. For the case $j\leq l$, we have that
\begin{align*}
     \bigl(\Lambda_k(P)B_k\bigr)_{(l-1)r+i,j}&=\sum_{h=j}^{l}\binom{p_i}{l-h+1}c_{h-1,j-1}\\
    &=
    \sum_{h=j}^{l}\sum_{m=1}^{l-h+1}
    c_{l-h+1,m}\,c_{h-1,j-1}\frac{p_i^m}{m!}\\
    &=
    \sum_{m=1}^{l-j+1}\frac{p_i^m}{m!}
    \sum_{h=j}^{l-m+1}
    c_{h-1,j-1}\,c_{l-h+1,m}\\
    &=
    \sum_{m=1}^{l-j+1}\frac{p_i^m}{m!}
    \sum_{h=0}^{l-m}
    c_{h,j-1}\,c_{l-h,m}
    \end{align*}
    where we did the change of variables $h\leftarrow h-1$ in the last equality, as well as using the fact that $c_{h,j-1}=0$ for $h<j-1$.

Using \cref{VdM-ID}, we get that the previous quantity equals
    $$  \sum_{m=1}^{l-j+1}
    c_{l,m+j-1}\frac{p_i^m}{m!}=
    \sum_{h=j}^{l}
    c_{l,h}\frac{p_i^{h-j+1}}{(h-j+1)!}=\bigl(\left(A_k \otimes I_r \right) \widetilde{\Lambda}_k(P)\bigr)_{(l-1)r+i,j}.
$$
  Then, the proposition follows from the definition of $\ws_k(P)$, \cref{lambdas-equality}, and the fact that the first row of $B_k$ is $\mathbf{e_1}^T$.
\end{proof}

\subsection{Polynomial multicorrelation sequences in totally ergodic Weyl systems}
Now we turn to the analysis for the structure of polynomial multicorrelation sequences in Weyl systems. We will provide a ``polynomial Fourier expansion'' of such sequences into functions of the form $n\mapsto e(q(n))$, revealing the natural frequencies $q$ of the system. It turns out that these frequencies are precisely the Weyl polynomials. We start with standard Weyl systems and the case in which $f_0,\ldots,f_r$ are characters of the underlying compact abelian group. We recall that for $x\in \T$ we denote $e(x):=\exp\left(2\pi \mathbf{i} x\right)$.
\begin{prop}\label{Weyl-poly-in-standard}
    Let $(X=\T^d,\mu,T)$ be a $d$-step standard Weyl system with rotation by $\beta\in \T$. Let $P=\{p_1,\ldots,p_r\}$ be a set of essentially distinct integral polynomials with zero constant term. For $v\in \Z^d$ denote $\psi_v(x)=e(v_1x_1+\ldots+v_dx_d)$ with $x\in X$. Then, for every $v^0,\ldots,v^r\in \Z^d$ either the sequence
    $$n\mapsto \int \psi_{v^0} T^{p_1(n)}\psi_{v^1}\cdots T^{p_r(n)}\psi_{v^r}d\mu $$
    is zero for all but finitely many $n\in \Z$, or there is $q\in \ws_d(P)$ such that
    $$\int \psi_{v^0} T^{p_1(n)}\psi_{v^1}\cdots T^{p_r(n)}\psi_{v^r}d\mu  =
        e(q(n)\beta) .$$
\end{prop}
\begin{proof}
  Denote $$\mathbf{v}=(v_{1}^1,\ldots,v_{1}^r,v_{2}^1,\ldots,v_{2}^r,\cdots,v_{d}^1,\ldots,v_{d}^r)^T ,$$ 
  and $p_0=0$. We compute the integral: 
     \begin{align*}
 &\int \psi_{v^0} T^{p_1(n)}\psi_{v^1}\cdots T^{p_r(n)}\psi_{v^r}d\mu\\
     &=\int  e\Big( \sum_{k=0}^{r} \sum_{i=1}^{d} v^k_{i} (p_k^{[0]}(n)x_{i} + \cdots +p_k^{[i-1]}(n) x_{1}+ p_k^{[i]}(n)\beta )   \Big) dm_{\T^d}(x)\\
     &=  e\Bigg(\Big( \sum_{k=1}^{r}\sum_{i=1}^{d} p_k^{[i]}(n)v_{i}^k\Big)\cdot \beta\Bigg)\cdot  \prod_{i=1}^d \int  e\Big( \sum_{k=0}^r  (p_k^{[0]}(n) v_{i}^k +p_k^{[1]}(n) v_{i+1}^k+...+p_k^{[d-i]}(n)v_{d}^k )   x \Big)dm_{\T}(x).
     \end{align*}
Consider the condition 
     \begin{equation}\label{condition_0}
\sum_{k=0}^r  (p_k^{[0]}(n) v_{i}^k +p_k^{[1]}(n) v_{i+1}^k+...+p_k^{[d-i]}(n)v_{d}^k ) =0, ~\forall i\in [d] .
\end{equation} 
Then, we obtain that 
\begin{equation}
    \int \psi_{v^0} T^{p_1(n)}\psi_{v^1}\cdots T^{p_r(n)}\psi_{v^r}d\mu= \begin{cases}
              e\Bigg( (\mathbf{v}^T\Lambda_d\mathbf{e}_1)\beta\Bigg)& \text{ if $n$ satisfies \cref{condition_0}}\\
           
            ~~~0 & \text{ if not}
        \end{cases}.
\end{equation}       
As the condition in \cref{condition_0} is polynomial, it either holds only in a finite set or holds for every $n\in \Z$. In the latter case, it is equivalent to 
\begin{equation}\label{condition}
     v^0=-\sum_{k=1}^r  v^k \text{ and } \mathbf{v} \in \Z^r\times \Rspa{\Lambda_{d-1}}^\perp.  
\end{equation}
This way, whenever \cref{condition} holds, we have that $q=\mathbf{v}^T\Lambda_d\mathbf{e}_1 \in \ws_d(P)$. Hence, 
$$n\mapsto \int \psi_{v^0} T^{p_1(n)}\psi_{v^1}\cdots T^{p_r(n)}\psi_{v^r}d\mu  $$
is either zero for all but finitely many $n\in \Z$, or there is $q\in  \ws_d(P)$ such that 
$$\int \psi_{v^0} T^{p_1(n)}\psi_{v^1}\cdots T^{p_r(n)}\psi_{v^r}d\mu= e(q(n)\beta), $$
concluding. 
\end{proof}
\cref{Weyl-poly-in-standard} gives an alternative definition of the Weyl polynomials in terms of weighted ergodic averages.
\begin{prop}\label{Weyl-Poly-alternative-def}
    Let $P=\{p_1,\ldots,p_r\}$ be essentially distinct integral polynomials with zero constant term, and $h$ be an integral polynomial with zero constant term. Then $h\notin \ws_d(P)$ if and only if for every $d$-step standard Weyl system, irrational number $\alpha\in \R\setminus \Q$, and functions $f_0,\ldots,f_r\in L^\infty$:
    $$\lim_{N\to \infty}\frac{1}{N} \sum_{n=1}^N e(h(n)\alpha)\int f_0T^{p_1(n)}f_1\cdots T^{p_r(n)} f_r d\mu=0. $$
\end{prop}
\begin{proof}
    Assume that $h\notin \ws_d(P)$ and let $(X=\T^d,\mu,T)$ be a $d$-step standard Weyl system with rotation by $\beta\in \R\setminus \Q$. Let $f_0,\ldots,f_r\in L^\infty(\mu)$. For each $k\in \{0,\ldots,r\}$ by writing $f_k$ in its Fourier expansion we can assume without loss of generality that $f_k= \psi_{v^k}$ for some $v^k\in \Z^d$. If the integral $\int \psi_{v^0} T^{p_1(n)}\psi_{v^1}\cdots T^{p_r(n)}\psi_{v^r}d\mu$ is zero for all but finitely many $n\in \N$ we are done. If not, by \cref{Weyl-poly-in-standard} there is $q\in \ws_d(P)$ such that 
    $$\int \psi_{v^0} T^{p_1(n)}\psi_{v^1}\cdots T^{p_r(n)}\psi_{v^r}d\mu= e(q(n)\beta). $$
    Then
   $$        \frac{1}{N} \sum_{n=1}^N e(h(n)\alpha)\int f_0T^{p_1(n)}f_1\cdots T^{p_r(n)} f_r d\mu =  \frac{1}{N} \sum_{n=1}^N e(h(n)\alpha+q(n)\beta ). $$
We have two cases. On one hand, if $\{1,\alpha,\beta\}$ are rationally independent, then the polynomial $n\mapsto h(n)\alpha+q(n)\beta$ has irrational leading coefficient, and thus by Weyl theorem 
   $$ \lim_{N\to \infty}  \frac{1}{N} \sum_{n=1}^N e(h(n)\alpha+q(n)\beta )=0.$$
On the other hand, if $\{1,\alpha,\beta\}$ are rationally dependent, then there are rational numbers $m_1,m_2\in \Q$ such that $\beta= m_1\alpha+ m_2$. Since $h\notin \ws_d(P)$, the rational polynomial $h+m_1q$ is not constant. Hence, the polynomial $n\mapsto (h(n)+m_1 q(n))\alpha+m_2q(n)$ has at least one irrational coefficient. Then, by Weyl's equidistribution theorem we get that 
 $$\lim_{N\to \infty}  \frac{1}{N} \sum_{n=1}^N e(h(n)\alpha+q(n)\beta )=\lim_{N\to \infty }\frac{1}{N} \sum_{n=1}^N e((h(n)+m_1 q(n))\alpha+m_2q(n))= 0.$$
For the other direction, assume that $h\in \ws_d(P)$. Let $(X=\T^d,\mu,T)$ be a $d$-step standard Weyl system with rotation by $\alpha\in \R\setminus \Q$. Consider $u^1,\ldots,u^r\in \R^d$ such that the vector
$$\mathbf{u}:=(u_{1}^1,\ldots,u_{1}^r,u_{2}^1,\ldots,u_{2}^r,\cdots,u_{d}^1,\ldots,u_{d}^r)^T $$ 
satisfies
\begin{equation}\label{eq-poly}
    (-h,0,\ldots,0)= \mathbf{u}^T\Lambda_d.
\end{equation}
It is direct that we can assume that $u^1,\ldots,u^r\in \Q^d$, since $h$ and the polynomials in $P$ have rational coefficients. Indeed, the system of equations given by the coefficients of the polynomials in \cref{eq-poly} is a linear system of equations over $\Q$, and thus the existence of solutions in $\R$ implies the existence of solutions in $\Q$.

In particular, there is $m\in \N$ such that if for each $i\in [r]$ we set $v^i:=m u^i=(mu_1^i,mu_2^i,\ldots, m u_d^i)$, then $v^i$ has integer coordinates. Set $v^0=-\sum_{k=1}^r v^k $. Then for each $N\in \N$
$$  \frac{1}{N} \sum_{n=1}^N e(h(n) (m\alpha))\int \psi_{v^0} T^{p_1(n)}\psi_{v^1}\cdots T^{p_r(n)}\psi_{v^r}d\mu = \frac{1}{N} \sum_{n=1}^N e(mh(n)\alpha) e(-mh(n)\alpha) =1, $$
concluding.
\end{proof}
\cref{Weyl-Poly-alternative-def} is particularly useful to deduce algebraic properties of the $\R$-vector space of Weyl polynomials, such as the following.
\begin{prop}\label{some-properties-weyl-poly}
    Let $p_1,\ldots,p_r,h$ be essentially distinct integral polynomials with zero constant term and $d\in \N$.
    \begin{enumerate}
        \item  We have that $ \ws_d(p_1,\ldots,p_r)= \ws_d(p_1-p_r,\ldots,p_{r-1}-p_r,-p_r).$
        \item  For every $m\in \N$, $i\in \N_0$, and 
        $M\in \N$ divisible enough so that the polynomials
        $$\frac{h(Mn +i )-h(i)}{m} ,\frac{p_1(Mn +i )-p_1(i)}{m},\ldots, \frac{p_r(Mn +i )-p_r(i)}{m} $$
        are integral, we have that $h\in \ws_d(\{p_1,\ldots,p_r\})$ if and only if
    $$ \frac{h(Mn +i )-h(i)}{m} \in \ws_d\left(\left\{\frac{p_1(Mn +i )-p_1(i)}{m},\ldots,\frac{p_r(Mn +i )-p_r(i)}{m}\right\}\right). $$
    \end{enumerate}
\end{prop}
\begin{proof}
    Statement $\textbf{1.}$ comes from the fact that 
    $$\int f_0 T^{p_1(n)}f_1\cdots T^{p_r(n)}f_r d\mu = \int f_r T^{p_1(n)-p_r(n)}f_1\cdots T^{p_{r-1}(n)-p_r(n)}f_{r-1} T^{-p_r(n)}f_0 d\mu, $$
    and thus by \cref{Weyl-Poly-alternative-def}, for a polynomial $q:\Z\to \Z$ we have that $q\in \ws_d(p_1,\ldots,p_r)$ if and only if $q\in \ws_d(p_1-p_r,\ldots,p_{r-1}-p_r,-p_r)$.
    
    For statement $\textbf{2.}$ let us first prove the case when $m=1$ and $M=1$. This case comes from noticing that for each $(X,\mu,T)$ a $d$-step standard Weyl system, and  $\alpha\in \R\setminus \Q$, we have that
    $$\lim_{N\to \infty}\frac{1}{N} \sum_{n=1}^N e(h(n)\alpha)\int f_0T^{p_1(n)}f_1\cdots T^{p_r(n)} f_r d\mu=0$$
    for all $f_0,\ldots,f_r\in L^\infty(\mu)$ if and only if 
    $$\lim_{N\to \infty}\frac{1}{N} \sum_{n=1}^N e((h(n+i)-h(i))\alpha)\int g_0T^{p_1(n+i)-p_1(i)}g_1\cdots T^{p_r(n+i)-p_r(i)} g_r d\mu=0, $$
     for all $g_0,\ldots,g_r\in L^\infty(\mu)$, as we can take $g_0=f_0$ and $g_j= T^{p_j(i)}f_j$ for each $j\in \{1,\ldots,r\}$ and multiply the expression by $e(-h(i)\alpha)$.

    To prove the general case, it is enough to prove the case when $i=0$. Notice that by \cref{general-props} \textbf{4.}, $h\in \ws_d(\{p_1,\ldots,p_r\})$ if and only if $h(Mn)\in \ws_d(\{p_1(Mn),\ldots,p_r(Mn)\} )$, where the reverse direction comes from the fact that if $h(Mn)=q(Mn)$ for every $n\in \N$, with $q\in \ws_d(P)$, then necessarily $h=q$ and thus $h\in \ws_d(P)$. On the other hand, by \cref{change-of-basis}, $h(Mn)\in \ws_d(\{p_1(Mn),\ldots,p_r(Mn)\} )$ if and only if there is $\mathbf{v}=(v_1^T,\ldots,v_d^T)^T\in \Q^{r\cdot d}$ such that for each $n\in \N$, $$(h(Mn),0,\ldots,0)= \mathbf{v}^T\tilde{\Lambda}_{d}(\{p_1(Mn),\ldots,p_r(Mn)\})$$ or equivalently
     $$\left(\frac{h(Mn)}{m},0,\ldots,0\right)= (v_1^T,mv_2^T,\ldots,m^{d-2} v_{d-1}^T,m^{d-1}v_d^T) \tilde{\Lambda}_d\left(\left\{\frac{p_1(Mn)}{m},\ldots,\frac{p_r(Mn)}{m}\right\}\right). $$
     This implies that  $h\in \ws_d(\{p_1,\ldots,p_r\})$ if and only if
     $$ \frac{h(Mn  )}{m} \in \ws_d\left(\left\{\frac{p_1(Mn  )}{m},\ldots,\frac{p_r(Mn  )}{m}\right\}\right), $$
      concluding the argument.
\end{proof} 
The following result extends \cref{Weyl-poly-in-standard} to all totally ergodic Weyl systems, giving \cref{teo-C} as a special case. 
\begin{theo}\label{standard-decomposition}
    Let $(X,\mu,T)$ be a connected $d$-step Weyl system with Kronecker factor given by $x\in \T^s\mapsto x+\beta$, where $\beta=(\beta_1,\ldots,\beta_s)\in \T^s$ for some $s\in \N$. Let $P=\{p_1,\ldots,p_r\}$ be a set of essentially distinct integral polynomials with zero constant term. Then for functions $f_0,\ldots,f_r\in L^\infty(\mu)$ there are a sequence $(c_l)_{l\in \N} \in \ell^2(\N)$ with $$\norm{(c_l)_{l\in \N} }_{\ell^2(\N)}\leq\prod_{k=0}^{r}||f_k||_2$$ and integral polynomials $\{q_{l,j}\}_{l\in \N,j\in [s]}\subseteq \ws_d(P)$ such that
    \begin{equation}\label{Hidden-freq}
        \liminf_{L\to \infty} \limsup_{n\to \infty} \left| \int_X f_0T^{p_1(n)}f_1 \cdots T^{p_r(n)}f_r d\mu -\sum_{l=1}^L c_l \cdot e\left(q_{l,1}(n)\beta_1+\cdots+q_{l,s}(n)\beta_s\right)\right |=0.
    \end{equation}
\end{theo}
\begin{proof}
Using \cref{FKLemma} and possibly passing to an extension, it suffices to consider the case in which $(X,\mu,T)$ is a product of $s$ many standard $d$-step Weyl systems, with $X=\T^{sd}$ and $T=S_1\times\cdots \times S_s$ where $(\T^d,S_j)$ is a standard Weyl system such that for $\beta_j\in \R\setminus \Q$,
    $$S_j(x_1,...,x_{d})= (x_1+\beta_j,x_2+x_1,...,x_{d}+ x_{d-1}).$$
    
Let $\epsilon\in(0,1)$ be fixed and $L\in \N$ to be determined. Let $f_0,\ldots,f_r\in L^\infty(\T^{sd})$ be functions which we will assume, without loss of generality, bounded by $1$. We notice that $L^2(\T^{sd})$ has the orthonormal basis given by
$$x=(x_1,\ldots,x_s)\in (\T^d)^s\mapsto \psi_{l_1,\ldots,l_s}(x)=\psi_{l_1}(x_1)\cdots  \psi_{l_s}(x_s)$$
with $l_1,\ldots l_s\in \Z^d$.

For $L\in \N$, and $f\in L^2(\mu)$, denote
$$\mathcal{S}_L(f) =\sum_{l_1,\ldots,l_s\in [-L,L]^d} \langle f ,\psi_{l_1,\ldots,l_s}\rangle  \psi_{l_1,\ldots,l_s}. $$

For $k\in \{0,\ldots,r\}$, as $f_k\in L^\infty(\mu)$ we can have that the Fourier series of $f_k$ converges to $f_k$ in $L^{2r}(\mu)$-norm (see for example \cite[Theorem II.1.5]{Katznelson_2004}). Thus, we can take $L\in \N$ big enough so that 
$$\norm{f_k-\mathcal{S}_L(f_k) }_{L^{2r}(\mu)}\leq \frac{1}{(r+1)2^r} \epsilon, $$
where $[-L,L]=\{-L,\ldots, -1,0,1,\ldots, L\}$. Notice that in particular $\norm{\mathcal{S}_L(f_k)}_{L^{2r}(\mu)}\leq 2$ for each $k\in \{0,\ldots,r\}$. In this way, by Cauchy-Schwarz and the fact that $f_0,\ldots,f_r$ are bounded by $1$, for each $n\in \N$ we have that using Hölder's inequality
\begin{align*}
   &\left|\int_X f_0T^{p_1(n)}f_1 \cdots T^{p_r(n)}f_r d\mu- \int \mathcal{S}_L(f_0)T^{p_1(n)}\mathcal{S}_L(f_1) \cdots T^{p_r(n)}\mathcal{S}_L(f_r) d\mu  \right| \\
   &\leq \sum_{k=0}^r \norm{f_k -  \mathcal{S}_L(f_k)}_{L^{2}(\mu)} \cdot \left( \prod_{i<k} \norm{\mathcal{S}_L(f_i)}_{L^{2r}(\mu)}  \right)\cdot \left(  \prod_{i>k} \norm{f_i}_{L^{2r}(\mu)}  \right) \\
   & \leq  \frac{\epsilon}{(r+1)2^r} \sum_{k=0}^r 2^k \leq \epsilon.  
\end{align*}
 Hence, taking limsup in $n$ and then limsup when $L\to \infty$ and $\epsilon\searrow 0$, we get that
\begin{align}
   &\lim_{L\to \infty}\sup_{n\in \N} \Bigg| \int_X f_0T^{p_1(n)}f_1 \cdots T^{p_r(n)}f_r d\mu\\ \label{double-line-eq}
   &-\sum_{\substack{l_1^k,\ldots,l_s^k\in [-L,L]^d\\k=0,\ldots,r}} \left(\prod_{k=0}^r\langle f_k,\psi_{l_1^k,\ldots,l_s^k}\rangle \right)\cdot \int  \psi_{l_1^0,\ldots,l_s^0} T^{p_1(n)}\psi_{l_1^1,\ldots,l_s^1}\cdots T^{p_r(n)}\psi_{l_1^r,\ldots,l_s^r} d\mu\Bigg|=0.  
\end{align}

To finish, we observe that on the one hand, we have the following identity
$$\int  \psi_{l_1^0,\ldots,l_s^0} T^{p_1(n)}\psi_{l_1^1,\ldots,l_s^1}\cdots T^{p_r(n)}\psi_{l_1^r,\ldots,l_s^r} d\mu =\prod_{j=1}^s \int  \psi_{l_j^0} S_j^{p_1(n)}\psi_{l_j^1}\cdots S_j^{p_r(n)}\psi_{l_j^r} dm_{\T^d}, $$
and if this product is not zero in a cofinite set, then by \cref{Weyl-poly-in-standard} there are $q_1,\ldots,q_s\in \ws_d(P)$ such that 
$$\prod_{j=1}^s \int  \psi_{l_j^0} S_j^{p_1(n)}\psi_{l_j^1}\cdots S_j^{p_r(n)}\psi_{l_j^r} dm_{\T^d}= \prod_{j=1}^s e(q_j(n) \beta_j)= e(q_1(n)\beta_1+\cdots+q_s(n)\beta_s). $$
On the other hand, we have that
\begin{align*}
\left( \sum_{\substack{l_1^k,\ldots,l_s^k\in [-L,L]^d\\k=0,\ldots,r}} \left|\prod_{k=0}^r\langle f_k,\psi_{l_1^k,\ldots,l_s^k}\rangle \right|^2 \right)^{1/2}&= \prod_{k=0}^r \left( \sum_{l_1,\ldots,l_s\in [-L,L]^d} |\langle f_k,\psi_{l_1,\ldots,l_s}\rangle|^2 \right)^{1/2}  \\ 
&\leq \prod_{k=0}^r \left( \sum_{l_1,\ldots,l_s\in \Z^d} |\langle f_k,\psi_{l_1,\ldots,l_s}\rangle|^2 \right)^{1/2}  \\
   & =
   \prod_{k=0}^r ||f_k||_2.
\end{align*}
Therefore, enumerating accordingly the non eventually zero  terms in \cref{double-line-eq} finishes the proof.
\end{proof}

\cref{standard-decomposition} allows us to prove the following extension of \cref{Weyl-Poly-alternative-def}.
 \begin{theo}\label{orthogonality-in-Weyl}
Let $d,r \in \mathbb{N}$ and let $(X,\mu,T)$ be a $d$-step Weyl system. Let $P=\{p_1,\ldots,p_r\}$ be a set of essentially distinct integral polynomials with zero constant term, $h$ an integral polynomial with zero constant term such that $h\notin \ws_d(P)$, and $\alpha\in \R\setminus \Q$. If $g: \mathbb{T} \rightarrow \mathbb{C}$ is Riemann integrable with $\int_\T g dm_{\T}=0$ and $f_0,f_1, \ldots, f_{r} \in L^{\infty}(\mu)$, then 
$$
\begin{aligned}
 \lim_{N\to \infty} \frac{1}{N} \sum_{n=1}^N g(h(n)\alpha) \int f_0T^{p_1(n)} f_1  \cdots  T^{p_r(n)} f_{r} d\mu=0.
\end{aligned}
$$
 \end{theo}
\begin{proof}
     Without loss of generality, let us assume that the functions $f_0, \ldots, f_r$ are continuous and bounded by $1$. Since any Riemann integrable function can be approximated from above and below by continuous functions, and those in turn can be uniformly approximated by trigonometric polynomials by the Stone–Weierstrass theorem, we also assume that $g(t)=e(kt)$, for some $k\in \Z\setminus\{0\}$. Moreover, by replacing $h$ by $kh$, one can assume that $k=1$.
     
     We prove first the case in which $X$ is connected. By \cref{standard-decomposition} there is $\beta\in \T^s$, a sequence $(c_l)_{l\in \N} \in \ell^2(\N)$ with $\norm{(c_l)_{l\in \N} }_{\ell^2(\N)}\leq 1$, integral polynomials $\{q_{l,j}\}_{l\in \N,j\in [s]}\subseteq \ws_d(P)$ such that for $\epsilon>0$ there are $M,\overline{N}\in \N$ such that
    $$ \left|\int f_0 T^{p_1(n)}f_1 \cdots T^{p_r(n)}f_r d\mu  -  \sum_{l=1}^M c_l \cdot e\left( \sum_{j=1}^s q_{l,j}(n)\cdot \beta_j \right)
    \right|\leq \epsilon,~ \forall n\geq \overline{N}. $$
    Thus, for $N\in \N$ with $\overline{N}/N\leq \epsilon$ we have that
    \begin{align*}
        &\left|\frac{1}{N}\sum_{n=1}^N e(h(n)\alpha)\cdot \int f_0(x)f_1(T^{p_1(n)}x)\cdots f_r(T^{p_r(n)}x) d\mu(x) \right|\\
        &\leq  2\epsilon + \left|\frac{1}{N}\sum_{n=\overline{N}}^N\sum_{l=1}^M c_l \cdot e\left(h(n)\alpha+ \sum_{j=1}^s q_{l,j}(n)\cdot \beta_j \right)\right|\\
        &\leq  2\epsilon + \sum_{l=1}^M\left|\frac{1}{N}\sum_{n=\overline{N}}^N  e\left(h(n)\alpha+ \sum_{j=1}^s q_{l,j}(n)\cdot \beta_j \right)\right|.
    \end{align*}
    Notice that for each $l\in \{1,\ldots,M\}$, $ h(n)\alpha + \sum_{j=1}^s q_{l,j}(n)\beta_j$ is a non constant polynomial with at least one irrational coefficient. Indeed, if $\{1,\beta_1,\ldots,\beta_s,\alpha\}$ are rationally independent, then the statement is direct. If not, as $\{1,\beta_1,\ldots,\beta_s\}$ are rationally independent, we have that there are rational numbers $m_0,m_1,\ldots,m_s$ such that 
    $$\alpha=m_0+m_1\beta_1+\cdots+m_s\beta_s. $$
    If the polynomial $h(n)\alpha + \sum_{j=1}^s q_{l,j}(n)\beta_j$ does not have an irrational coefficient, then
    $$h(n)\alpha + \sum_{j=1}^s q_{l,j}(n)\beta_j= m_0h(n)+\sum_{j=1}^s (q_{l,j}(n)+m_jh(n))\beta_j\in \Q,$$
    which implies that for each $j\in [s]$ and $n\in \N$,
    $$q_{l,j}(n)+m_jh(n)=0. $$
The fact that $h\notin \ws_d(P)$ yields $m_j=0$ for each $j\in [s]$, and thus $\alpha\in \Q$, which is a contradiction. Thus, by Weyl equidistribution theorem
$$\limsup_{N\to \infty} \left|\frac{1}{N}\sum_{n=1}^N e(h(n)\alpha)\cdot \int f_0 T^{p_1(n)}f_1 \cdots T^{p_r(n)}f_r d\mu\right|\leq 2\epsilon. $$
As $\epsilon>0$ was arbitrary, we conclude that 
$$\lim_{N\to \infty} \frac{1}{N}\sum_{n=1}^N e(h(n)\alpha)\cdot \int f_0 T^{p_1(n)}f_1 \cdots T^{p_r(n)}f_r d\mu =0 $$
finishing the proof in the connected case.

For the general case, if $X_0$ denotes the connected component of $e_X$ on $X$, then by compactness there is $m \in \N$ such that $T^mX_0=X_0$ and the sets $\{T^{i}X_0\}_{i=1}^m$ are pairwise disjoint. In particular, we have that $\mu=\frac{1}{m}\sum_{j=1}^m T^j \mu_{X_0}$. Take $M\in \N$ to be divisible enough, so that $m| M$ and 
$$\frac{h(Mn +i )-h(i)}{m} ,\frac{p_1(Mn +i )-p_1(i)}{m},\ldots, \frac{p_r(Mn +i )-p_r(i)}{m} $$
are integral polynomials for each $i\in \{1,\ldots,M\}$. This is possible due to the fact that each $p_j$ has rational coefficients, and that for each $k\in \N$, one can write $(Mn+i)^k-i^k=M\cdot q(n)$ where $q$ has integer coefficients.  

Then, we have that 
\begin{align*}
    &\lim_{N\to \infty }\frac{1}{N}\sum_{n=1}^N e(h(n)\alpha)\cdot \int_X f_0 T^{p_1(n)}f_1 \cdots T^{p_r(n)}f_r d\mu \\
    &=  \lim_{N\to \infty }\frac{1}{N/M}\sum_{n\leq N/M} \frac{1}{M}\sum_{i=1}^M e(h(Mn+i)\alpha)\cdot \int_{X} f_0 T^{p_1(Mn+i)}f_1 \cdots T^{p_r(Mn+i)}f_r d\mu\\
    &=  \lim_{N\to \infty }\frac{1}{N/M}\sum_{n\leq N/M} \frac{1}{M}\sum_{i=1}^M \frac{1}{m}\sum_{j=1}^m e(h(Mn+i)\alpha)\cdot \int_{X_0} T^jf_0 T^{p_1(Mn+i)}T^jf_1 \cdots T^{p_r(Mn+i)}T^jf_r d\mu_{X_0}.
\end{align*}
Thus, it is enough to show that for $i\in \{1,\ldots,M\}$ and for all $f_0,\ldots,f_r\in L^\infty(\mu_{X_0})$, 
\begin{equation}\label{eq-n}
   \lim_{N\to \infty}\frac{1}{N}\sum_{n=1}^{N} e((h(Mn+i)-h(i))\alpha)\cdot \int_{X_0} f_0 T^{p_1(Mn+i)-p_1(i)}f_1 \cdots T^{p_r(Mn+i)-p_r(i)}f_r d\mu_{X_0}=0 . 
\end{equation}
 Notice that the polynomials $\{\frac{p_j(Mn+i)-p_j(i)}{m}\}_{j\in [r]}$ are essentially distinct integral polynomials. If we denote the transformation $S=T^m$, then proving \cref{eq-n} amounts to showing that 
 \begin{equation}\label{eq-n+1}
      \lim_{N\to \infty}\frac{1}{N}\sum_{n=1}^{N} e\left(\frac{h(Mn+i)-h(i)}{m}\cdot (m\alpha)\right)\cdot \int_{X_0} f_0 S^{\frac{p_1(Mn+i)-p_1(i)}{m}}f_1 \cdots S^{\frac{p_r(Mn+i)-p_r(i)}{m}}f_r d\mu_{X_0}  
 \end{equation}
 is equal to zero. On the other hand, the fact that \cref{eq-n+1} equals zero follows from the connected case, as $$\frac{h(Mn+i)-h(i)}{m}\notin \ws_d\left(\frac{p_1(Mn+i)-p_1(i)}{m},\ldots, \frac{p_r(Mn+i)-p_r(i)}{m} \right)$$ by \cref{some-properties-weyl-poly} and $(X_0,\mu_{X_0},S)$ is a totally ergodic Weyl system by \cref{connected}. 
\end{proof}
\begin{remark}
  By \cref{Weyl-Poly-alternative-def} the condition $h\notin \ws_d(P)$ is necessary for \cref{orthogonality-in-Weyl}.  
\end{remark}
As a direct consequence of \cref{standard-decomposition} and \cref{orthogonality-in-Weyl} we obtain \cref{teo-D}, which we restate for the convenience of the reader.
\begin{corxD}
Let $P=\{p_1,\ldots,p_r\}$ and $Q=\{q_1,\ldots,q_l\}$ be sets of distinct integral polynomials with zero constant term. Then $$\ws(P)\cap \ws(Q)=\{0\}$$ if and only if, for every totally ergodic Weyl systems $(X,\mu,T)$ and $(Y,\nu,S)$, and functions $f_0,\ldots,f_r\in L^\infty(X)$ and $g_0,\ldots,g_l\in L^\infty(Y)$, the sequences
$$ n\mapsto  \int f_0T^{p_1(n)} f_1  \cdots  T^{p_r(n)} f_{r} d\mu, \text{ and } n\mapsto \int g_0S^{q_1(n)} g_1  \cdots  S^{q_l(n)} g_{l} d\nu $$
are asymptotically uncorrelated.
\end{corxD}
\begin{proof}
  We prove the statement by proving each implication.\\
  $(\Longrightarrow)$ Without loss of generality, assume that all the functions involved $f_0,\ldots,f_r,g_0,\ldots,g_l$ are bounded by $1$. Denote by $\beta\in\T^{s_1}$ and $\gamma\in\T^{s_2}$ the elements inducing the rotation in the Kronecker factors of $(X,T)$ and $(Y,S)$ respectively. Let $\epsilon\in(0,1)$. By \cref{standard-decomposition}, there exist sequences $(c_m)_m,(C_m)_m\in \ell^2(\N)$ with $\norm{(c_m)_m}_{\ell^2(\N)},\norm{(C_m)_m}_{\ell^2(\N)}\leq 1$, integral polynomials $\{ u_{m,j}\}_{m\in \N, j\in [s_1]}\subseteq \ws(P) $ and $\{ v_{m,j}\}_{m\in \N, j\in [s_2]}\subseteq \ws(Q) $, and $L_1,L_2,\overline{N}\in\N$ such that for all $n\geq \overline{N}$:
  \begin{equation}\label{approximation-for-P}
       \left| \int_X f_0T^{p_1(n)}f_1 \cdots T^{p_r(n)}f_r d\mu -\sum_{m=1}^{L_1} c_m \cdot e\left(u_{m,1}(n)\beta_1+\cdots+u_{m,s_1}(n)\beta_{s_1}\right) \right|\leq \epsilon,
  \end{equation}
  and
   \begin{equation}\label{approximation-for-Q}
       \left| \int g_0S^{q_1(n)} g_1  \cdots  S^{q_l(n)} g_{l} d\nu -\sum_{m=1}^{L_2} C_m \cdot e\left(v_{m,1}(n)\gamma_1+\cdots+v_{m,s_2}(n)\gamma_{s_2}\right) \right|\leq \epsilon.
  \end{equation}
  Notice that \cref{approximation-for-P} and \cref{approximation-for-Q} also imply that for $n\geq \overline{N}$:
  $$   \left| \sum_{m=1}^{L_1} c_m \cdot e\left(u_{m,1}(n)\beta_1+\cdots+u_{m,s_1}(n)\beta_{s_1}\right) \right| \text{ and } \left|\sum_{m=1}^{L_2} C_m \cdot e\left(v_{m,1}(n)\gamma_1+\cdots+v_{m,s_2}(n)\gamma_{s_2}\right)\right| $$
  are bounded by $2$. 
  
  By \cref{approximation-for-P} and \cref{approximation-for-Q} we also have that up to an error of $3\epsilon$, the limit
  \begin{equation}\label{multicorrelation-eq}
      \lim_{N\to \infty} \frac{1}{N} \sum_{n=1}^N \left( \int f_0T^{p_1(n)} f_1  \cdots  T^{p_r(n)} f_{r} d\mu\right) \cdot \left(\int g_0S^{q_1(n)} g_1  \cdots  S^{q_l(n)} g_{l} d\nu  \right)  
  \end{equation}
  is equal to
  \begin{equation}\label{exp-with-exponential-averages}
     \sum_{m=1}^{L_1}\sum_{l=1}^{L_2} c_m C_l \lim_{N\to \infty} \frac{1}{N} \sum_{n=1}^N  e\left(\sum_{j=1}^{s_1}u_{m,j}(n)\beta_j\right) e\left(\sum_{j=1}^{s_2} v_{l,j}(n)\gamma_j\right) 
  \end{equation}
  For each $m\in [L_1]$ and $l\in [L_2] $, denote $$ u_m(n):= \sum_{j=1}^{s_1}u_{m,j}(n)\beta_j, \text{ and } v_l(n):=\sum_{j=1}^{s_2} v_{l,j}(n)\gamma_j.$$
  Since $\ws(P)\cap \ws(Q)=\{0\}$ and both $\ws(P)$ and $\ws(Q)$ have bases in $\Q[n]$, we have that the union of these bases is rationally independent. In particular, we have that $u_m+v_l \in \Q[n]$ if and only if $u_m\in \Q[n]$ and $v_l\in \Q[n]$. Since the systems $(X,\mu,T)$ and $(Y,\nu,S)$ are ergodic, $(1,\beta_1,\ldots,\beta_{s_1})$ and $(1,\gamma_{1},\ldots,\gamma_{s_2})$ have each rationally independent coordinates. Thus $u_m+v_l \in \Q[n]$ is equivalent to $u_m=v_l=0$. This way, the expression in \cref{exp-with-exponential-averages} equals
  \begin{align*}
      &= \sum_{m=1}^{L_1}\sum_{l=1}^{L_2} \ind{u_m=0}  c_m \ind{v_l=0}C_l\\
      &=  \left( \sum_{m=1}^{L_1}\ind{u_m=0}  c_m \right) \left( \sum_{l=1}^{L_2}  \ind{v_l=0}C_l\right)\\
      &= \left( \sum_{m=1}^{L_1}  c_m \lim_{N\to \infty} \frac{1}{N} \sum_{n=1}^N  e\left(u_m(n)\right)   \right) \left( \sum_{l=1}^{L_2}   C_l \lim_{N\to \infty} \frac{1}{N} \sum_{n=1}^N  e\left(v_l(n)\right) \right).
  \end{align*}  
  By \cref{approximation-for-P} and \cref{approximation-for-Q}, up to an error of $3\epsilon$  the previous expression is equal to
\begin{equation}\label{product-of-averages}
        \left( \lim_{N\to \infty} \frac{1}{N} \sum_{n=1}^N\int f_0T^{p_1(n)} f_1  \cdots  T^{p_r(n)} f_{r} d\mu\right) \cdot \left(\lim_{N\to \infty} \frac{1}{N} \sum_{n=1}^N\int g_0S^{q_1(n)} g_1  \cdots  S^{q_l(n)} g_{l} d\nu  \right).
\end{equation}
Taking $\epsilon\searrow 0$, we conclude that the expression in \cref{multicorrelation-eq} equals the expression in \cref{product-of-averages}, concluding the implication. 
  \\
  $(\Longleftarrow)$ Let $u\in \ws(P)\cap \ws(Q)$. Assume without loss of generality that $u$ is nonzero and integral. This is possible due to the fact that both $\ws(P)$ and $\ws(Q)$ have a basis of integral polynomials, thus if the only integral polynomial in $\ws(P)\cap \ws(Q)$ is $0$, that would imply that $\ws(P)\cap \ws(Q)=\{0\}$.
  
  Observe that $u$ is not constant as it is not zero. For any irrational $\alpha\in \R\setminus \Q$, it follows from the proof of \cref{Weyl-poly-in-standard} that we can construct a system $(\T^d,m_{\T^d},T)$ and find functions $f_0,\ldots,f_r\in L^\infty(\mu)$ such that 
  $$\int f_0 T^{p_1(n)}f_1\cdots T^{p_r(n)} f_r d\mu = e(u(n)\alpha).   $$
  By hypothesis we thus have that for any totally ergodic Weyl system $(Y,\nu,S)$, and functions $g_0,\ldots,g_l\in L^\infty(Y)$ we have that 
  $$\lim_{N\to \infty}\frac{1}{N}\sum_{n=1}^N e(u(n)\alpha)\int g_0S^{q_1(n)} g_1  \cdots  S^{q_l(n)} g_{l} d\nu =0.  $$
  Then, by \cref{Weyl-Poly-alternative-def}, we have that $u\notin \ws(Q)$, which is a contradiction, concluding the implication and hence the theorem.
\end{proof}

It can be shown that \cref{teo-D} only requires one of $(X,\mu,T)$ or $(Y,\nu,S)$ to be totally ergodic. However, the statement fails if neither is totally ergodic, as the following example reveals.
\begin{ex}
    Consider the group $\Z_2=\{0,1\}$ and the transformation $T:\Z_2\to \Z_2$ given by $T(x)=x+1\text{ mod } 2$. For $P=\{n\}$ and $Q=\{n^2\}$ we notice that for each $n\in \N$
    $$ \int \ind{\{1\}} (x)\cdot T^n\ind{\{1\}} (x) dm_{\Z_2}(x)= \frac{1}{2}\ind{2\N}(n)$$
    and 
    $$ \int \ind{\{0\}} (x)\cdot T^{n^2}\ind{\{1\}} (x) dm_{\Z_2}(x)= \frac{1}{2}\ind{2\N+1}(n).$$
    Consequently,
    $$\lim_{N\to \infty}\frac{1}{N}\sum_{n=1}^N \int \ind{\{1\}} (x)\cdot T^n\ind{\{1\}} (x) dm_{\Z_2}(x) = \frac{1}{4} $$
    and 
    $$\lim_{N\to \infty}\frac{1}{N}\sum_{n=1}^N  \int \ind{\{0\}} (x)\cdot T^{n^2}\ind{\{1\}} (x) dm_{\Z_2}(x) = \frac{1}{4}. $$
    Nonetheless
    $$\lim_{N\to \infty}\frac{1}{N}\sum_{n=1}^N \left(\int \ind{\{1\}} (x)\cdot T^n\ind{\{1\}} (x) dm_{\Z_2}(x)\right) \left( \int \ind{\{0\}} (x)\cdot T^{n^2}\ind{\{1\}} (x) dm_{\Z_2}(x)\right)=0.$$
\end{ex}
We believe that \cref{orthogonality-in-Weyl} can be generalized to all systems and \cref{teo-D} to all totally ergodic systems. However, some technical complications arise when trying to lift these results to nilsystems. Hence, we conjecture the following.
\begin{conj}\label{conjecture-correlation-sequences}
Let $(X,\mu,T)$ be an invertible measure preserving system. Let $P=\{p_1,\ldots,p_r\}$ be a set of essentially distinct integral polynomials with zero constant term, $h$ an integral polynomial with zero constant term such that $h\notin \ws(P)$, and $\alpha\in \R\setminus \Q$. If $g: \mathbb{T} \rightarrow \mathbb{C}$ is Riemann integrable and $f_0,f_1, \ldots, f_{r} \in L^{\infty}(\mu)$, then 
$$
\begin{aligned}
 \lim_{N\to \infty} \frac{1}{N} \sum_{n=1}^N \left( g(h(n)\alpha)- \int_\T g dm_{\T}\right) \int f_0T^{p_1(n)} f_1  \cdots  T^{p_r(n)} f_{r} d\mu=0.
\end{aligned}
$$
\end{conj}
\begin{conj}\label{conj-D}
Let $P=\{p_1,\ldots,p_r\}$ and $Q=\{q_1,\ldots,q_l\}$ be sets of essentially distinct integral polynomials with zero constant term. Then $$\ws(P)\cap \ws(Q)=\{0\}$$ if and only if, for every pair of totally ergodic systems $(X,\mu,T)$ and $(Y,\nu,S)$, and functions $f_0,\ldots,f_r\in L^\infty(X)$ and $g_0,\ldots,g_l\in L^\infty(Y)$, the sequences
$$ n\mapsto  \int f_0T^{p_1(n)} f_1  \cdots  T^{p_r(n)} f_{r} d\mu, \text{ and } n\mapsto \int g_0S^{q_1(n)} g_1  \cdots  S^{q_l(n)} g_{l} d\nu $$
are asymptotically uncorrelated.
\end{conj}

\section{Recurrence in Weyl systems}\label{sec4}
In this section we prove \cref{teo-A}. We start by giving sufficient conditions for polynomial recurrence in products of standard Weyl systems via recurrence in the Kronecker factor of such systems. In the sequel, we extend these sufficient conditions to all Weyl systems. For this, we will need to reduce to the connected case in order to use \cref{FKLemma} to lift the conditions from standard Weyl systems. Then, for proving that this condition is indeed necessary we show that every set of $P$-topological recurrence in Weyl systems is a set of $Q$-topological recurrence in Kronecker systems, where $Q$ is a basis of Weyl polynomials associated to $P$. This closes the series of implications as every set of $P$-measurable recurrence is a set of $P$-topological recurrence. 

\subsection{Sufficient conditions for measurable polynomial recurrence}
In this subsection we prove that for a given product of standard Weyl systems $(X,\mu,T)$ and a set of integral polynomials $P=\{p_1,\ldots,p_r\}$ with zero constant term, if a set $R\subseteq \N$ is a set of $\ws(P)$-recurrence for the Kronecker factor for $(X,\mu,T)$, then $R$ is a set of $P$-measurable recurrence for $(X,\mu,T)$. The meaning of being a set of recurrence for an infinite set of polynomials is defined next.
\begin{defn}
   Given a set of real polynomials $P$ (not necessarily finite) and a class of systems $\sC$, we will say that $R\subseteq\N$ is a set of $P$-measurable recurrence for the class $\sC$ of systems if for every finite subset of integral polynomials $Q\subseteq P$, $R$ is a set of $Q$-measurable recurrence for the class $\sC$.
\end{defn}
We will need the following property which shows that the Ramsey property is preserved for sets of $P$-recurrence for Kronecker systems, where $P$ may be an infinite family of polynomials.
\begin{prop}\label{Ramsey-Property}
    Let $P$ be a nonempty set of polynomials. Let $R\subseteq \N$ be a set of $P$-recurrence for Kronecker systems. Assume we partition $R$ into sets $R_1,\ldots,R_L$ for some $L\in \N$. Then, there is $i\in \{1,\ldots,L\}$ such that $R_i$ is a set of $P$-recurrence for Kronecker systems. 
\end{prop}
\begin{proof}
    Assume by contradiction that for each $i\in \{1,\ldots,L\}$, there are finite sets $Q_1,\ldots,Q_L\subseteq P$ of integral polynomials such that $R_i$ is not a set of $Q_i$-recurrence for Kronecker systems. This is, for some $\epsilon>0$, there are $d\in \N$, and $\alpha\in \T^d$, and $N_1,\ldots,N_L\in \N$ such that, for each $i\in \{1,\ldots,L\}$ and $n\in R_i$ with $n\geq N_i$, there exists $q\in Q_i$ such that 
    \begin{equation}\label{eq-rec-contradiction}
        \norm{q(n)\alpha} >\epsilon.
    \end{equation}
    However, since $R$ is a set of $P$-recurrence for Kronecker systems, there is $n\in R=R_1\sqcup \cdots \sqcup R_L$ with $n\geq \max\{N_1,\ldots,N_L\}$ such that for each $q\in Q_1\cup\cdots \cup Q_L$:
    $$\norm{q(n)\alpha}\leq \epsilon, $$
    which contradicts \cref{eq-rec-contradiction}. We thus conclude that there is $i\in \{1,\ldots,L\}$ such that $R_i$ is a set of $P$-recurrence for Kronecker systems.
\end{proof}

We now introduce all the machinery that we need from IP limits. Following \cite{Furstenberg81,Furstenberg_Katznelson78,Bergelson96} a set $S\subseteq \N$ is called \textit{IP} if there is an infinite sequence $(x_n)_{n\in \N}\subseteq \N$ such that $S$ contains all finite sums of $(x_n)_{n\in \N}\subseteq \N$:
$$\left\{ \sum_{i\in \alpha} x_i \mid  \alpha\in \mathcal{F}  \right\}\subseteq S,$$
where $\mathcal{F}$ denotes the nonempty finite subsets of $\N$. An \textit{IP}$^*$ set is a set which has nonempty intersection with every IP set. For $\alpha,\beta\in \sF$ we write $\alpha<\beta$ if $\max(\alpha)<\min(\beta)$. We define an \textit{IP-ring} as
$$\sF'=\left\{\bigcup_{i\in \beta} \alpha_i \mid \beta\in \sF \right\},$$
where $(\alpha_i)_{i\in \N}$ is an increasing sequence of elements in $\sF$. Observe that each IP-ring $\sF'$ is in bijection to $\sF$ via $\varphi:\sF\to \sF'$, $\varphi(\beta)=\bigcup_{i\in \beta}\alpha_i$. Thus, any element indexed by $\sF'$ can be seen as an element indexed by $\sF$. For an $\sF$-sequence $(x_\alpha)_{\alpha\in \sF}$ in a topological space $X$, let $x\in X$ and let $\sF'$ be an IP-ring. One writes
$$\iplim_{\alpha\in \sF'} x_\alpha =x $$
if for any neighborhood $U$ of $x$ there exists $\alpha_U\in \sF'$ such that for any $\alpha\in \sF'$ with $\alpha>\alpha_U$ one has $x_\alpha\in U$.

We will use a special case of a notable result from Bergelson and McCutcheon on uniform bounds in recurrence along IP limits, which we state now.
\begin{theo}[Cf. {\cite[Theorem 1.3]{Bergelson_McCutcheon00}}]\label{theo-1.3}
    Let $(n_\alpha)_{\alpha\in \sF}$ be an IP-set and $\{p_1,\ldots,p_r\}$ be a set of integral polynomials with zero constant term and degree at most $d$. Let $(X,\sB,\mu,T)$ be a measure preserving system. For every $A\in \sB$ with $\mu(A)>0$, there exist an IP-ring $\sF'$ and a real number $c=c(A, r, d)>0$ such that
    \begin{equation*}
       \iplim_{\alpha\in \sF'}\mu(A\cap T^{-p_1(n_\alpha)}A \cap \cdots \cap T^{-p_r(n_\alpha)}A) \geq c.
    \end{equation*}
\end{theo}
The reason why this result becomes useful to us is the necessity of controlling simultaneously the frequencies given by \cref{Hidden-freq}. For this, we will need the set of natural numbers $n\in \N$ in which such frequencies are near to $0$ to define an IP$^*$-set, which is given by the following proposition.
\begin{prop}[Cf. {\cite[Theorem 2.19]{Furstenberg81}}]\label{IPstarness}
Let $q_1,\ldots,q_r$ be integral polynomials with zero constant term. Then for each $\alpha\in \R^d$ and $\epsilon>0$ the set
$$\{n\in \N \mid \norm{q_i(n)\alpha}_{\T^d}<\epsilon ,\forall i\in \{1,\ldots,r\}\} $$
is $IP^*$.
\end{prop}

\begin{cor}\label{cor-IP-prop}
   Let $P=\{p_1,\ldots,p_r\}$ and $\{q_m : m\in \N\}$ be sets of integral polynomials with zero constant term, and suppose that $\deg p_i\leq d$ for $1\leq i\leq r$. Let $(\beta_m)_{m\in\N}\subseteq\T$, and let $(X,\sB,\mu,T)$ be a measure-preserving system. Then, for every $A\in \sB$ with $\mu(A)>0$ there are $c(A,r,d)>0$, an IP-set $(n_{\alpha})_{\alpha}$, and an IP-ring $\sF'$ such that
    \begin{equation*}
       \iplim_{\alpha\in \sF'}\mu(A\cap T^{-p_1(n_\alpha)}A \cap \cdots \cap T^{-p_r(n_\alpha)}A) \geq c
    \end{equation*}
    and for every $m\in \N$
    \begin{equation*}
         \iplim_{\alpha\in \sF'} ||q_{m}(n_\alpha)\beta_m|| =0.
    \end{equation*}
\end{cor}
\begin{proof}
    Define for each $k,M\in \N$ the set
    \begin{equation*}
        \mathcal{A}_{k,M}:=\{ n\in \N \mid \forall m\leq M, \norm{q_m(n)\beta_m}\leq \frac{1}{k}\}.
    \end{equation*}
    By \cref{IPstarness} the set $\mathcal{A}_{k,M}$ is IP$^*$. Let $FS(\{y_i\}_{i\in\N})$ be an IP-set contained in $\mathcal{A}_{1,1}$. Set $x_1=y_1$. Suppose we have constructed $x_1,\ldots,x_k$ such that there is an IP-set $FS(\{z_i\}_{i\in \N})$ satisfying that for each $i\in [k]$:
    $$FS(\{x_j\}_{j=i}^{k}\cup \{z_i\}_{i\in \N})\subseteq \mathcal{A}_{i,i}. $$
    By Hindman's theorem, IP-sets are partition regular. Thus, we can find an IP-set $FS(\{u_i\}_{i\in \N})$ contained in $FS(\{z_i\}_{i\in \N})\cap \mathcal{A}_{k+1,k+1}$ (observe that it cannot be contained in the complement since $\mathcal{A}_{k+1,k+1}$ is IP$^*$). Set $x_{k+1}=u_1$ and take $FS(\{u_{i}\}_{i\geq 2})$ as the IP set for the next iteration. We notice that now we have that for each $i\in [k+1]$:
    $$FS(\{x_j\}_{j=i}^{k+1}\cup \{u_i\}_{i\geq 2})\subseteq \mathcal{A}_{i,i}. $$
    This way, we have constructed an IP-set $FS(\{x_j\}_{j\in \N})$ such that for each $i\in \N$:
    $$FS(\{x_j\}_{j\geq i})\subseteq \mathcal{A}_{i,i}.$$
    In particular, there is an IP-set $(n_\alpha)_{\alpha}$ defined by
    $$n_\alpha=\sum_{j\in \alpha} x_j, $$
    and an IP-ring $\sF'=\sF$ such that for each $m\in \N$
    \begin{equation*}
           \iplim_{\alpha\in \sF'} ||q_{m}(n_\alpha)\beta_m|| =0.
    \end{equation*}
    In consequence, by \cref{theo-1.3} we can find an IP-ring $\sF^{''}\subseteq \sF^{'}$ such that 
        \begin{equation*}
       \iplim_{\alpha\in \sF^{''}}\mu(A\cap T^{-p_1(n_\alpha)}A \cap \cdots \cap T^{-p_r(n_\alpha)}A) \geq c,
    \end{equation*}
    and for every $m\in \N$
    \begin{equation*}
         \iplim_{\alpha\in \sF^{''}} ||q_{m}(n_\alpha)\beta_m|| =0.
    \end{equation*}
\end{proof}
Now we put all the ingredients together to get the desired sufficient conditions for polynomial recurrence in standard Weyl systems.
\begin{theo}\label{sufficient-conditions}
    Let $R\subseteq \N$ and $P=\{p_1,\dots,p_r\}$ be a set of integral polynomials with zero constant term. Then, for each ergodic product of $d$-step standard Weyl systems $(X,\mu,T)$ if $R$ is a set of $\ws_d(P)$-recurrence for the Kronecker factor of $(X,\mu,T)$, then $R$ is a set of $P$-measurable recurrence for $(X,\mu,T)$.
\end{theo}
\begin{proof}
We denote $T=S_1\times\cdots\times S_{s}$ such that for each $j=1,\ldots, s$, $(\T^{d},m_{\T^{d}} ,S_j)$ is a standard Weyl system with 
\begin{equation*}
    S_j(x)=(x_1+\beta_j,x_2+x_1,...,x_{d}+ x_{d-1}).
\end{equation*}

Let $A$ be a measurable set of positive measure on $X$. Using \cref{standard-decomposition} with $f_0=\cdots=f_r=\ind{A}$ we obtain there are a sequence $(c_l)_{l\in \N}\in \ell^2(\N)$ with $$\norm{(c_l)_{l\in \N} }_{\ell^2(\N)}\leq\prod_{k=0}^{r}||\ind{A}||_2\leq 1$$ and integral polynomials $\{q_{l,j}\}_{l\in \N,j\in[s]}\subseteq\ws_d(P)$ such that if we denote $q_l=(q_{l,1},\ldots,q_{l,s})$ for each $l\in \N$, then for each $\epsilon>0$ there are $N_0\in \N$ and $M\in \N$ with
    \begin{equation}\label{approximation-with-Weyl-frequencies}
       \left|\mu(A\cap T^{-p_1(n)}A\cap \cdots \cap T^{-p_r(n)}A) -\sum_{l =1}^M c_l \cdot e\left(q_l(n)^T\beta\right)\right|< \epsilon,~ \forall n\geq N_0.
    \end{equation}
Let $\delta>0$ be given by \cref{cor-IP-prop} such that there is an IP-set $(n_\alpha)_{\alpha}$ and an IP-ring $\mathcal{F}'$ such that
$$\iplim_{\alpha\in \sF'}\mu(\bigcap_{k=0}^r T^{-p_k(n_\alpha)}A)\geq 2\delta, \text{ and } \iplim_{\alpha\in \sF'} ||q_{l,j}(n_\alpha)\beta_j|| =0, $$
for each $l\in \N$ and $j\in [s]$, where $p_0=0$. Take $0<\epsilon<\delta/4$ and $n\geq N_0$ such that $\mu(\bigcap_{k=0}^r T^{-p_k(n)}A)\geq \delta$ and $||q_l(n)^T\beta|| \leq  \epsilon/(2\pi M)$ for every $l\in \{1,\ldots, M\}$, where $N_0,M\in \N$ come from \cref{approximation-with-Weyl-frequencies} for such $\epsilon$. In this way, for such an $n\in \N$, we get the upper bound
 \begin{align*}
      \mu(\bigcap_{k=0}^r T^{-p_k(n)}A)  &\leq \epsilon  + \left|\sum_{l =1}^Mc_l \cdot e\left(q_l(n)^T\beta\right)\right|\\
      &\leq \epsilon+ \sum_{l =1}^M|c_l| \cdot |e\left(q_l(n)^T\beta\right)-1| +\left|\sum_{l =1}^Mc_l \right|\\
        &\leq \epsilon+ \frac{\epsilon}{M}\sum_{l =1}^M|c_l|  +\left|\sum_{l =1}^Mc_l \right|\\
      &\leq 2\epsilon+\left|\sum_{l =1}^Mc_l  \right|,
 \end{align*}
from which we conclude 
\begin{equation*}
  \left|\sum_{l =1}^M c_l \right| \geq \delta - 2\epsilon \geq \delta/2 .
\end{equation*}
 
That being so, if $R$ is a set of $\ws_d(P)$-recurrence for the Kronecker factor, we can take arbitrarily large $n \in R$ such that $n\geq N_0$ and $||q_l(n)^T\beta|| \leq  \epsilon/(2\pi M)$ for every $l\in\{1,\ldots,M\}$, yielding 
\begin{align*}
 \mu(\bigcap_{k=0}^r T^{-p_k(n)}A)  & \geq  -\epsilon  + |\sum_{l =1}^Mc_l \cdot e(q_l(n)^T\beta)|\\
    &\geq -\epsilon + \Big| \sum_{l =1}^Mc_l \Big|-\Big|\sum_{l =1}^Mc_l \cdot \Bigg(e\Big(q_l(n)^T\beta\Big)-1\Bigg) \Big| \\
    &\geq -2\epsilon+ \delta/2>0.
 \end{align*}
 Therefore, $R$ is a set of $P$-measurable recurrence for $(X,\mu,T)$. 
\end{proof}
\begin{remark}
    As the subgroup of integral polynomials in $\ws_d(\{p_1,\ldots,p_r\})$ has a 
    $\Z$-basis $\{q_i\}_{i\in I}$, we can replace $\ws_d(P)$-recurrence with $\{q_i\}_{i\in I}$-recurrence in the statement of \cref{sufficient-conditions}.
\end{remark}

\subsection{Characterization of topological and measurable recurrence}
In this subsection, we prove the main theorem of this section, which will give \cref{teo-A} as a corollary. 
\begin{theo}\label{Equivalences-Weyl}
    Let $d,r\in \N$, $R\subseteq \N$, and $P=\{p_1,...,p_r\}$ be a set of integral polynomials with zero constant term. Then, the following are equivalent:
    \begin{enumerate}
        \item $R$ is a set of $P$-measurable recurrence for $d$-step Weyl systems.
        \item $R$ is a set of $P$-topological recurrence for $d$-step Weyl systems.
        \item $R$ is a set of $\ws_d(P)$-recurrence for Kronecker systems. 
    \end{enumerate}
\end{theo}
\begin{proof}[Proof of \cref{Equivalences-Weyl}]
($\textbf{1.}\Longrightarrow \textbf{2.}$) This implication is a special case of \cref{measure-implies-topological}.

($\textbf{2.}\Longrightarrow \textbf{3.}$) Let $Q=\{q_1,\ldots,q_t\}$ be a basis of integral polynomials of $\ws_d(P)$. Let $R\subseteq \N$ be a set of $P$-topological recurrence for $d$-step Weyl systems and $\beta=(\beta_1,\ldots,\beta_s)\in [0,1)^s$. We need to prove that for each $\epsilon>0$ there is $n\in R$ such that for each $i\in \{1,\ldots,t\}$, $\norm{q_i(n)\beta}_{\T^s}<\epsilon. $ By definition of the Weyl space, for each $i\in \{1,\ldots,t\}$ there is $v_i\in \Q^{d\cdot r}$ such that $(q_i(n),0,\ldots,0)=v_i^T\Lambda_d(n)$. Let $\ell \in \N$ such that $w_i:=\ell\cdot v_i$ has integer coordinates for all $i\in \{1,\ldots,t\}$. Consider the affine system $(\T^{sd},m_{\T^{sd}},T)$ where $T=S_1\times\cdots\times S_{s}$, and for each $j=1,\ldots, s$, we define for $x\in \T^d$:
\begin{equation*}
    S_j=(x_1+\beta_j/\ell,x_2+x_1,...,x_{d}+ x_{d-1}).
\end{equation*}
Consider the minimal system $(X,\mu,T)$ where $X=\overline{\{T^n\vec{0} : n\in \Z\}},$ $\vec{0}\in \T^{sd}$ is the vector full of zeros, and $\mu$ is the Haar measure on $X$. Then, by \cite[Theorem 9, Chapter 11]{Host_Kra_nilpotent_structures_ergodic_theory:2018}, $(X,\mu,T)$ is a $d$-step Weyl system, as the underlying nilpotent Lie group $G_X$ of $X$ is such that $G_X\subseteq G$, where $G$ is the underlying nilpotent Lie group in the representation of $(\T^{sd},m_{\T^{sd}},T)$. Thus $G_X$ is still at most $d$-step with abelian connected component. 

Denote $U=B_{\T^{sd}}(0,\epsilon)\cap X\subseteq \T^{sd}$, where we use the maximum norm in $\T^{s\cdot  d}$ (i.e. $x\in U$ if and only if $||x_i||_{\T}<\epsilon$ for each $i\in \{1,\ldots,sd\}$). Hence, there are arbitrarily large $n\in R$ such that
\begin{equation*}
  U\cap T^{-p_1(n)} U\cap \cdots \cap T^{-p_r(n)}U\neq \emptyset.  
\end{equation*}
Consequently, there is $x\in U$ such that after reordering coordinates accordingly 
\begin{equation*}
  \left( \begin{array}{c}
         x \\
         \vdots \\
        x
    \end{array}\right) + \left[  \Lambda_{d}(n)\left( \begin{array}{c}
       \beta_j /\ell \\
          x_{j,1}\\
         \vdots \\
         x_{j,d-1}
    \end{array}\right) \right]_{j=1}^{s}  =     \left( \begin{array}{c}
         T^{p_1(n)}x \\
         \vdots \\
        T^{p_r(n)}x
    \end{array}\right) \in B_{\T^{(sd) \cdot r}}(0,\epsilon),
\end{equation*}
and so, for every $j\in \{1,\ldots,s\}$ we have that 
\begin{equation*}
    \Lambda_{d}(n)\left( \begin{array}{c}
        \beta_j/\ell \\
          x_{j,1}\\
         \vdots \\
         x_{j,d-1}
    \end{array}\right)  \in B_{\T^{d\cdot r}}(0,2\epsilon).
\end{equation*}
We recall that for each $i\in \{1,\ldots,t\}$ there is $w_i\in \Z^{d\cdot r}$ such that $(\ell \cdot q_i(n),0,\ldots,0)= w_i^T\Lambda_{d}(n)$, which gives that for each $j\in \{1,\ldots,s\}$
\begin{equation*}
 q_i(n) \beta_j= (\ell \cdot q_i(n),0,\ldots,0) \left( \begin{array}{c}
        \beta_j/\ell  \\
          x_{j,1}\\
         \vdots \\
         x_{j,d-1}
    \end{array}\right) = 
    w_i^T\Lambda_d(n)\left( \begin{array}{c}
        \beta_j/\ell \\
          x_{j,1}\\
         \vdots \\
         x_{j,d-1}
    \end{array}\right) \in B_{\T}(0,C\epsilon),
\end{equation*}
for a constant $C>0$ that only depends on $w_1,\ldots,w_t$, and thus, only depending on $Q$. We conclude that there is $n\in R$ such that for every $i\in \{1,\ldots,t\}$,
\begin{equation*}
    ||q_i(n)\beta||_{\T^s}\leq C(Q)\epsilon
\end{equation*}
which provides the desired conclusion by replacing $\epsilon$ by $\epsilon/C(Q)$.

($\textbf{3.}\Longrightarrow \textbf{1.}$) Let $(X=G/\Gamma,\mu,T)$ be a $d$-step Weyl system and $R\subseteq \N$ be a set of $\ws_d(P)$-recurrence for Kronecker systems. We know by \cite[Corollary 7 and 8 in Chapter 11]{Host_Kra_nilpotent_structures_ergodic_theory:2018} that there is $l\in \N$ such that if $X_0=G_0\Gamma/\Gamma$ is the connected component of $e_X:=e_G\Gamma$ in $X$ where $e_G$ is the identity of $G$, then $\{T^{i}X_0\}_{i=1}^l$ partitions $X$. In particular, if $\mu_{X_0}$ is the Haar measure of $X_0$, then $\mu=\frac{1}{l}\sum_{i=0}^{l-1}T^i\mu_{X_0}$.

 Let $L\in \N$ such that $l| L$ and such that for each $i\in [r]$, $j\in \{0,\ldots,L-1\}$, and $n\in \Z$, $p_i(Ln+j)-p_i(j)$ has integer coefficients which are divisible by $l$. Now, we partition the set $R$ into residue classes mod $L$:
$$R=\bigcup_{j=0}^{L-1}R_j $$
 where $R_j=R\cap (L\Z+j)$. By \cref{Ramsey-Property}, there is $j\in \{0,\ldots,L-1\}$ such that $R_j$ is a set of $\ws_d(P)$-recurrence for Kronecker systems. In particular, there is $n\in \Z$ such that $Ln+j\in R$ and for each $i\in [r]$, $p_i(Ln+j)\equiv 0$ (mod $l$). Since $l|L$, we conclude that for each $m\in \Z$, $p_i(Lm+j)$ is divisible by $l$, since 
 $$p_i(Ln+j) \equiv p_i(Lm+j)\quad \text{(mod $l$)}.$$  
Thus, for each $n\in R_j$ and $i\in [r]$, $l|p_i(n)$. Observe that this already covers the case in which $X_0$ is a point and $X$ is finite. Thus, we will assume from now on that $X_0$ is not a point (and thus, it is isomorphic to a torus).

Now we will create another $d$-step Weyl system $(Y,\nu,S)$ such that $Y$ is connected, and $(Y,\nu,S^l)\cong (X_0,\mu_{X_0},T^l)$. For this, we write $Tx=Ux+b$ and let $D\in \N$ such that $X_0$ is identified with $\T^D$. Call $\varphi:X_0\to \T^D$ this group isomorphism. Let $A$ be the integer matrix such that $A\varphi(x)=\varphi(Ux)$ for each $x\in X_0$. Since $U$ is $d$-step unipotent, so is $A$. Define the matrix $H=I+A+\cdots+A^{l-1}$, and notice that $H$ is invertible as $A$ is unipotent. Indeed, a simple calculation shows that $H$ has determinant $l^D$. In particular, multiplication by $H$ is surjective in $\T^D$. Thus, we can find $c\in \T^D$ such that $Hc=\varphi(b_l)$, where $b_l=T^l\vec{0}$. Then, we define in $\T^D$ the transformation $Sy=Ay+c$. We observe that for $x\in X_0$, $$S^l(\varphi(x))= A^l \varphi(x) + Hc =\varphi(U^lx)+ \varphi(b_l)= \varphi(T^lx). $$
On one hand, we get that $\varphi:(X_0,T^l)\to (Y,S^l)$ is a conjugacy. Moreover, since $(X_0,\mu_{X_0},T^l)$ is ergodic, as $(X,\mu,T)$ is ergodic, we have that $S^l$ is ergodic, and hence $S$ is ergodic. On the other hand, $(Y,S)$ is a connected $d$-step Weyl system. If $A\in\mathcal{B}(X)$ is a set with positive measure, then there is $m\in \{1,\ldots,l\}$ such that $A \cap T^{m}X_0$ has positive $\mu$ measure. Thus, by $T$-invariance of the measure $\mu$, $A_m:=T^{-m}A \cap X_0$ has positive measure for $\mu_{X_0}$. In particular, $B:=\varphi(A_m)$ has positive $\nu$-measure, where $\nu=\varphi_*\mu_{X_0}$ is going to be the Haar measure on $Y$ by unique ergodicity. We notice that it is enough to prove then that there is $n\in R_j$ such that 
$$\nu(B \cap S^{-p_1(n)}B \cap \cdots\cap S^{-p_r(n)}B)>0. $$
Indeed, since for $n\in R_j$ we have that $l| p_i(n)$ for each $i\in [r]$, then we have that: 
\begin{align*}
    0<\nu(B \cap S^{-p_1(n)}B \cap \cdots\cap S^{-p_r(n)}B)&=\nu(B \cap (S^l)^{-p_1(n)/l}B \cap \cdots\cap (S^l)^{-p_r(n)/l}B)\\
    &= \mu_{X_0}( A_m \cap (T^l)^{-p_1(n)/l}A_m \cap \cdots\cap (T^l)^{-p_r(n)/l}A_m)\\
    &= \mu_{X_0}(A_m \cap T^{-p_1(n)}A_m \cap \cdots\cap T^{-p_r(n)}A_m)\\
    &\leq l \mu(A \cap T^{-p_1(n)}A \cap \cdots\cap T^{-p_r(n)}A).
\end{align*}
Moreover, by \cref{FKLemma} the system  $(Y=\T^D,\nu,S)$ is a factor of an ergodic product of standard $d$-Weyl systems. Hence, it is enough to prove that $R_j$ is a set of $P$-recurrence for the case in which $(Y=\T^D,\nu,S)$ is a product of standard Weyl systems. This follows from \cref{sufficient-conditions}, concluding the implication. 
\end{proof}

Omitting the restriction on the step of the Weyl system and taking a basis of integral polynomials $Q=\{q_1,\ldots,q_r\}$ of $\ws(P)$ in  \cref{Equivalences-Weyl} gives \cref{teo-A}.
   \begin{cor}[\cref{teo-A}]\label{teo-2}
    Let $r\in \N$ and $P=\{p_1,...,p_r\}$ be a set of essentially distinct integral polynomials with zero constant term. Then there exist rationally independent integral polynomials $Q=\{q_1,...,q_r\}$ such that for each $R\subseteq \N$, the following are equivalent:
    \begin{enumerate}
        \item  $R$ is a set of $\{p_1,\ldots,p_r\}$-measurable recurrence for Weyl systems,
        \item  $R$ is a set of $\{p_1,\ldots,p_r\}$-topological recurrence for Weyl systems.
        \item $R$ is a set of $\{q_1,\ldots,q_r\}$-recurrence for Kronecker systems. 
    \end{enumerate}
\end{cor} 
   Notice that $P$-recurrence in $W(P)$-step Weyl systems suffices to ensure $P$-recurrence in Weyl systems. This allows us to deduce \cref{teo-B}, generalizing the result from Host, Kra, and Maass.
   \begin{cor}[\cref{teo-B}]\label{s-implies-s+1}
       Let $r\in \N$ and $P=\{p_1,\ldots,p_r\}$ be a set of essentially distinct integral polynomials with zero constant term. Let $s\geq W(P)$. If $R\subseteq \N$ is a set of $P$-recurrence for $s$-step Weyl systems, then it is a set of $P$-recurrence for $(s+1)$-step Weyl systems.
   \end{cor}

\begin{remark}\label{remark-theo-B}
    The condition $s\geq W(P)$ in \cref{s-implies-s+1} cannot be weakened. Indeed, the case $W(P)=1$ is trivial, so let us assume that $W(P)\geq 2$. By \cref{Weyl-complexity-and-space} there is an integral polynomial $q\in \ws(P)$ such that $q\notin \ws_{W(P)-1}(P)$. Therefore, for an irrational number $\alpha\in \R$ we can define 
    $$R=\left\{ n\in \N \mid \norm{q(n)\alpha }>\frac{1}{4} \right\}, $$
    which is a set of $P$-recurrence for $(W(P)-1)$-step Weyl systems but not for $W(P)$-step Weyl systems by \cref{orthogonality-in-Weyl} and \cref{Equivalences-Weyl}. 
\end{remark}
Lastly, the following corollary describes how the Weyl polynomials completely control recurrence within Weyl systems.
\begin{cor}\label{characterization-of-recurrence-Weyl}
    Let $P=\{p_1,...,p_r\}$ and $Q=\{q_1,...,q_l\}$ be two sets of essentially distinct integral polynomials with zero constant terms. Then, $\ws(P)\subseteq \ws(Q)$ if and only if every set of $Q$-recurrence in Weyl systems is a set of $P$-recurrence in Weyl systems. 
    In particular, $P$-recurrence in Weyl systems and $Q$-recurrence in Weyl systems are in general position if and only if $\ws(P)\setminus \ws(Q)\neq \emptyset$ and $\ws(Q)\setminus \ws(P)\neq \emptyset$.
\end{cor}
\begin{proof}
    We prove it by double implication.
    
    $(\Longrightarrow)$ This implication is direct from \cref{teo-2}. Indeed, for a set $R\subseteq \N$ of $Q$-recurrence in Weyl systems, we have that $R$ is a set of $\ws(Q)$-recurrence for Kronecker systems. In particular, $R$ is a set of $\ws(P)$-recurrence for Kronecker systems, which by \cref{teo-2} is equivalent to $R$ being a set of $P$-recurrence for Weyl systems.

    $(\Longleftarrow)$ We prove this implication by contrapositive. Let $h\in \ws(P)\setminus \ws(Q)$ be an integral polynomial. Pick $\alpha\in \R\setminus \Q$ and set $$R=\Big\{ n\in \N \mid \{h(n)\alpha\}\in [1/4,3/4] \Big\} .$$ Clearly $R$ is not a set of $P$-recurrence for Weyl systems by \cref{Equivalences-Weyl}. However, $R$ is a set of $Q$-recurrence for Weyl systems. Indeed, let $(X,\mu,T)$ be a Weyl system and $A\subseteq X$ be a measurable set such that $\mu(A)>0$. It is easy to show that $R$ has density $1/2$. By \cref{orthogonality-in-Weyl}, we have that 
    \begin{align*}
       &\lim_{N\to \infty }\frac{1}{N}\sum_{n=1}^N \ind{R}(n) \cdot \mu(A\cap T^{-q_1(n)}A\cap \cdots \cap T^{-q_l(n)}A) \\
       &= \frac{1}{2}\lim_{N\to \infty }\frac{1}{N}\sum_{n=1}^N  \mu(A\cap T^{-q_1(n)}A\cap \cdots \cap T^{-q_l(n)}A), 
    \end{align*}
    where the positivity of the last expression comes from the polynomial Szemerédi theorem \cite{Bergelson_Leibman96}. This way, we conclude that $R$ is a set of $Q$-recurrence for Weyl systems.
\end{proof}

\begin{ex}\label{illustrative-ex}
Let $a,b,c\in \N$ be different positive integers. By a simple computation, we get: 
  \begin{itemize}
      \item $W(an,bn,cn)=3$ and $\ws(an,bn,cn)=\{ c_1n+c_2n^2+c_3n^3 : c_1,c_2,c_3\in \R \}$,
      \item $W(an,bn,cn^3)=2$ and $\ws(an,bn,cn^3)=\{ c_1n+c_2n^2+c_3n^3 : c_1,c_2,c_3\in \R \}$,
      \item $W(an,bn^2,cn^3)=1$ and $\ws(an,bn^2,cn^3)=\{ c_1n+c_2n^2+c_3n^3 : c_1,c_2,c_3\in \R \}$, and
      \item $W(n,2n,cn^2)=3$ and $\ws(n,2n,cn^2)=\{ c_1n+c_2n^2+c_3(n^3-\frac{c}{2}n^4) : c_1,c_2,c_3\in \R \}$.
  \end{itemize}
    In particular, we see that the recurrence schemes associated to $(an,bn,cn)$, $(an,bn,cn^3)$, and $(an,bn^2,cn^3)$ in Weyl systems are equivalent, but they are in general position with respect to $(n,2n,cn^2)$ in Weyl systems. Moreover, it is not hard to see that given $\alpha\in (0,1)\setminus \Q$ the sets 
    \begin{equation*}
       \Big\{n \mid ||n^3\alpha||>\frac{1}{4}\Big\}, \text{ and } \Big\{n \mid || \left(n^3-\frac{c}{2}n^4\right)\alpha||>\frac{1}{4}\Big\},  
    \end{equation*}
are sets of $(n,2n,cn^2)$-recurrence and $(an,bn,cn)$-recurrence in Weyl systems respectively, that are not sets of $(an,bn,cn)$-recurrence and $(n,2n,cn^2)$-recurrence in Weyl systems respectively.
\end{ex}
\begin{ex}\label{ex-double-linear-implies-quadratic}
    As stated in the introduction, every set of $\{n,2n\}$-recurrence for Weyl systems is a set of $\{n^2\}$-recurrence for Weyl systems. Indeed, by \cref{characterization-of-recurrence-Weyl} it is enough to notice that
    $$\ws(\{n^2\}) = \{ cn^2 :c\in \R\} \subseteq \{ c_1n+c_2n^2 : c_1,c_2\in \R\}= \ws(\{n,2n\}).  $$
\end{ex}
It is not possible to generalize \cref{characterization-of-recurrence-Weyl} to recurrence in all systems, as Griesmer \cite{GRIESMER_2024} showed that there is a set of $\{n,2n\}$ recurrence that is not a set of $\{n^2\}$-recurrence. Nevertheless, there is still hope for the converse direction, which we conjecture true for recurrence in all systems.
\begin{conj}\label{conj-general-position}
        Let $P=\{p_1,...,p_k\}$ and $Q=\{q_1,...,q_l\}$ be two sets of essentially distinct integral polynomials with zero constant terms such that $\ws(P)\setminus\ws(Q)\neq \emptyset$. Then, $Q$-measurable recurrence does not imply $P$-measurable recurrence.
      If in addition $\ws(Q)\setminus\ws(P)\neq \emptyset$ then
       $Q$-measurable recurrence and $P$-measurable recurrence are in general position.
\end{conj}
To finish, we show that \cref{conjecture-correlation-sequences} implies \cref{conj-general-position}.
\begin{prop}\label{C-implies-C}
     Assuming \cref{conjecture-correlation-sequences}, \cref{conj-general-position} holds.
\end{prop}
\begin{proof}
Let $p\in \ws(P)\setminus\ws(Q)$ be an integral polynomial and $\alpha\in \R\setminus\Q$ and define
\begin{equation*}
S:=\Big\{ n\in \N \mid \{p(n)\alpha\}\in [1/4,3/4] \Big\}.    
\end{equation*}
Clearly, $S$ is not a set of $P$-recurrence because it is not a set of $P$-recurrence for Weyl systems. Let us see that $S$ is a set of $Q$-averaging recurrence. Let $(X,\sB,\mu,T)$ be a measure-preserving system, and $A\subseteq X$ with $\mu(A)>0$, then by \cref{conjecture-correlation-sequences}:
\begin{align*}
        &\lim_{N\to \infty} \frac{1}{N}\sum_{n\leq N} \ind{S}(n) \mu(A\cap T^{-q_1(n)} A \cap \cdots \cap T^{-q_l(n)}A)\\
        &= \frac{1}{2}\lim_{N\to \infty} \frac{1}{N}\sum_{n\leq N} \mu(A\cap T^{-q_1(n)} A \cap \cdots \cap T^{-q_l(n)}A)>0,
\end{align*}
where the positivity of the last expression comes from the polynomial Szemerédi theorem \cite{Bergelson_Leibman96}.
\end{proof}
\begin{remark}
    In a previous version of this document, we showed how \cite[Conjecture 11.4]{Leibman10b} implied \cref{conjecture-correlation-sequences} and thus \cref{conj-general-position} by \cref{C-implies-C}. However, \cite[Conjecture 11.4]{Leibman10b} was recently proved false in \cite{ouyang2026connectednesspolynomialdiagonalorbit}. Nevertheless, the counterexample presented there does not disprove \cref{conjecture-correlation-sequences}, \cref{conj-D}, and \cref{conj-general-position}.
\end{remark}

\small{
\bibliographystyle{abbrv}
\bibliography{refs}

}
\bigskip
\noindent
Felipe Hernández\\
\textsc{{\'E}cole Polytechnique F{\'e}d{\'e}rale de Lausanne} (EPFL)\par\nopagebreak
\noindent
\href{mailto:felipe.hernandezcastro@epfl.ch}
{\texttt{felipe.hernandezcastro@epfl.ch}}

\end{document}